\documentclass[a4paper,11pt]{amsart}

\usepackage{amstext,amsfonts,amsthm,graphicx,amssymb,amscd,epsfig}
\usepackage{amsmath}
\usepackage[utf8]{inputenc}    
\usepackage[all]{xy}

\usepackage{mathrsfs}

\usepackage{enumitem}

\theoremstyle{definition}
\newtheorem{definition}{Definition}[section]
\newtheorem{ex}[definition]{Example}
\newtheorem{rem}[definition]{Remark}

\theoremstyle{plain}
\newtheorem{prop}[definition]{Proposition}
\newtheorem{lem}[definition]{Lemma}
\newtheorem{coro}[definition]{Corollary}
\newtheorem{teo}[definition]{Theorem}

\usepackage[usenames]{color}

\usepackage[
backend=biber,
style=numeric,
sorting=nyt,
giveninits=true,
doi=false,
date=year,
maxbibnames=99,
]{biblatex}
\DeclareFieldFormat
  [article,inbook,incollection,inproceedings,patent,thesis,unpublished]
  {title}{{\it #1\isdot}}
\DeclareFieldFormat{journaltitle}{#1\isdot}
\renewbibmacro{in:}{}

\usepackage{./custco}

\title[Augmentation of singularities]{Augmentation of singularities: $\mu/\tau$-type conjectures and simplicity}

\author{I. Breva Ribes, R. Oset Sinha}

\date{}

\address{Departament de Matem\`atiques,
Universitat de Val\`encia, Campus de Burjassot, 46100 Burjassot,
Spain}

\email{raul.oset@uv.es}

\email{igbreri@alumni.uv.es}

\thanks{Work of R. Oset Sinha partially supported by Grant PID2021-124577NB-I00 funded by MCIN/AEI/ 10.13039/501100011033 and by ``ERDF A way of making Europe"}


\begin{document}
\begin{abstract}
For function germs $g:(\mathbb C^n,0)\to (\mathbb C,0)$ it is well known that $1\leq\frac{\mu(g)}{\tau(g)}$ and it has recently been proved by Liu that $\frac{\mu(g)}{\tau(g)}\leq n$. We give an upper bound for the codimension of map-germs $f:(\mathbb C^n,0)\to (\mathbb C^p,0)$ given as augmentations of other map-germs with which we prove the analog to the first inequality (known as Mond's conjecture) for augmentations $h:(\mathbb C^n,0)\to (\mathbb C^{n+1},0)$. Furthermore, we show that the quotient given by the image Milnor number and the codimension of any augmentation in the pair of dimensions $(n,n+1)$ is less than $\frac{1}{4}(n+1)^2$ and prove the analog to the second inequality for map-germs with $n=1$ and augmentations with $n=2,3$. 
We then prove a characterization of when a map-germ is an augmentation, finding a counterexample for the characterization given by Houston.
Next, we give sufficient conditions for when the augmentation is independent of the choice of stable unfolding by studying different notions of equivalence of unfoldings.
Moreover, these results allow us to give sufficient conditions for the simplicity of an augmentation, providing context to locate the moduli for non-simple augmentations.
\end{abstract}

\maketitle

\section{Introduction}

Classification of singularities of map-germs $f:(\mathbb K^n,0)\to (\mathbb K^p,0)$ with $\bbk=\bbc$ or $\bbr$ has been an active field of research ever since Whitney classified stable maps from surfaces to the plane. Since then many landmark classifications have taken place, such as Arnold's classification of simple, uni-modal and bi-modal functions (\cite{arnold}) or Mond's classification under the group $\A$ (diffeomorphisms in source and target) of simple singularities of surfaces in 3-space (\cite{mond2en3}). It is also worth mentioning Goryunov's method for corank 1 maps germs with $n\geq p$ (\cite{gor84}), and the classifications of corank 1 simple germs for $n=p=2$ (Rieger in \cite{rieger2to2}), $n=p=3$ (Marar and Tari in \cite{marartari}) and $n=3,p=4$ (Houston and Kirk in \cite{houstonkirk}). Simplicity seems to have been a reasonable criteria to stop classifying singularities when $p>1$. While modality (i.e. lack of simplicity) has been thoroughly studied for functions (see \cite{varchenko,gabrielov}, for example) with striking connections to the topology ($\mu$-constant stratum) and algebraic invariants of the singularities, for the more complicated case of $p>1$ not so much can be said.

In the mentioned classifications for $p>1$, it is interesting to see that most simple germs are augmentations of other germs in lower dimensions.
\begin{definition}
Let $f\colon(\bbk^n,0)\to(\bbk^p,0)$ be a smooth map-germ with a 1-parameter stable unfolding $F(x,\lambda) = (f_\lambda(x),\lambda)$.
Given any smooth function $g\colon(\bbk^d,0)\to(\bbk,0)$, the {\it augmentation of $f$ by $F$ via $g$} will be the $\A$-class of the map-germ in $(\bbk^n\times\bbk^d,0)\to(\bbk^p\times\bbk^d,0)$ given by:
$$A_{F,g}(f)(x,z) = \left(f_{g(z)}(x),z\right)$$
We will say that $f$ is the base or the augmented map-germ, and that $g$ is the augmenting function.
\end{definition}
Understanding the simplicity (or lack of it) of these augmentations is the first step towards understanding modality in the general case. 
In a previous paper (\cite{augcod1morse}) we managed to characterize the simplicity of augmentations in certain cases: if $f$ has codimension 1, then $A_{F,g}(f)(x,z)$ is simple if and only if $g$ is simple, and if $g$ is a Morse function, then $A_{F,g}(f)(x,z)$ is simple if and only if $f$ is simple. 
However, what happens when we augment a germ $f$ of codimension 2 by a non-Morse function is much more subtle. 
For example the augmentation of the codimension 2 germ $(y^2,y^5)$ by the function $g(x)=x^3$ yields the germ $F_4=(x,y^2,y^5+x^3y)$ from Mond's list, which is simple, but the augmentation of the codimension 2 germ $11_5=(x,y^4+x^2y+xy^2)$ from Rieger's list by $g(z)=z^3$ is not simple.

It seems that a crucial part of understanding this phenomenon is analysing the independence of the augmentation from the choice of 1-parameter stable unfolding. While this follows from versality in the codimension 1 case (see \cite{coopermondwik} for $g$ Morse and \cite{augcod1morse} for any $g$ with isolated singularity), it is not trivial for higher codimension. 

On the other hand, some of the biggest open problems in singularities of mappings have to do with the codimension of a germ $f:(\bbc^n,0)\to(\bbc^{n+1},0)$ ($\aecod(f)$) and its image Milnor number $\mu_I(f)$ (the number of spheres in a stable perturbation). For functions $g:(\bbc^n,0)\to (\bbc,0)$ it is well-known that $\tau\leq \mu$ with equality if $g$ is quasihomogeneous. The analog result in mappings is only known for $n=1,2$ (\cite{mondbentwires,mondconj,vanstraten}) and is still open in general (see also \cite{bobadillanunopenafort}).

\smallskip
\emph{Conjecture 1 (Mond's conjecture):} Let $f:(\bbc^n,0)\to(\bbc^{n+1},0)$ be a smooth map germ, then $$\aecod(f)\leq \mu_I(f)$$ with equality if $f$ is quasihomogeneous.
\smallskip

Related to this problem, Liu has recently proved that for functions $g:(\bbc^n,0)\to (\bbc,0)$ we always have $\frac{\mu}{\tau}\leq n$. Obviously this bound is not sharp. Sharp bounds for
$\frac{\mu}{\tau}$ have been found when $n=2,3$ (see \cite{almiron1, almiron2}), although it is not known whether these bounds tend to infinity or not when $n$ tends to infinity. We can think of the analog problem for mappings, which was communicated to the authors by Peñafort Sanchis:

\smallskip
\emph{Conjecture 2:}  Let $f:(\bbc^n,0)\to(\bbc^{n+1},0)$ be a smooth map germ, then: $$\frac{\mu_I(f)}{\aecod(f)}\leq n+1.$$
\smallskip

In this paper we analyze all of these problems for augmentations. Section 2 gives the necessary notation. In \cite{houstonaug2} Houston gave a lower bound for the codimension of an augmentation. In Section 3 we give an upper bound and characterize the equality with the lower bound. Using these bounds and results by Houston (\cite{houstoninventiones}) and Giménez Conejero and Nu\~no-Ballesteros (\cite{gimeneznunoweak}) we prove Mond's conjecture (Conjecture 1) for augmentations (without the hypotheses used by the previous authors). We also give some partial converses for the equality. We then prove Conjecture 2 for any map germ with $n=1$, for any augmentations with $n=2$ and, under some hypothesis, for any augmentation with $n=3$. In general we prove that for augmentations the quotient in Conjecture 2 is less than or equal to $\frac{1}{4}(n+1)^2$.

In Theorem 4.4 of \cite{houstonaug2}, Houston gave a characterization for when a map germ is an augmentation. In Section 4 we construct a counterexample to his theorem and then prove a characterization that needs some extra conditions. We also give a geometric interpretation of this characterization.

Sections 5 to 8 introduce several notions of equivalence between unfoldings and relate them to the simplicity of the augmentations. In Section 5 we prove that the $\phi$-equivalence of 1-parameter stable unfoldings defined by Mancini and Ruas in \cite{manciniruas} implies the $\A$-equivalence of the augmentations and give sufficient conditions for an augmentation to be simple in terms of the number of $\phi$-equivalence classes of the unfoldings. For the case of non-simple augmentations we manage to locate the modality in certain strata of the versal unfolding. In Section 6 we define the notion of weak equivalence between unfoldings and show that it preserves substantiality (see Section 2 for definitions). We also give a sufficient condition for all the unfoldings to be weak equivalent. Then, in Section 7 we show that if all 1-parameter stable unfoldings are weak-equivalent and one of them is substantial, then they are all $\phi$-equivalent.

The equivalences of unfoldings and the sufficient conditions for them to happen, which depend on tangent vector fields to the discriminant of the unfolding, are interesting by themselves and may have applications in other contexts such as deformation theory.

We give plenty of examples throughout the paper which help understand the definitions and scope of the results presented. Moreover, in Section 8 we analyze the simplicity of certain augmentations when $n=p$ and thus give a proof of a conjecture we stated in \cite{augcod1morse} by giving a classification of all corank 1 simple augmentations from $\bbc^4$ to $\bbc^4$.

{{\it Acknowledgements:} The authors thank the Singularity group at the University of València for the constant support.
The authors also thank M. A. S. Ruas for the fruitful discussions and encouragement during their stay at ICMC, where some of the ideas from this paper took form.}

\section{Preliminaries}

Through most of this text, we will be working indistinctly  over $\bbc$ and $\bbr$, denoting any of them by $\bbk$.
In the real case, {\it smooth} will mean infinitely differentiable; in the complex case, it will mean holomorphic.
For $s\in\bbn$, denote by $\ofu_s$ the local ring of germs of smooth functions $(\bbk^s,0)\to(\bbk,0)$ with maximal ideal $\mfr_s$ and $\theta_s = \ofu_s\times\overset{s}\cdots\times\ofu_s$ the $\ofu_s$-module of vector fields in $\bbk^s$.

Let $f\colon(\bbk^n,0)\to(\bbk^p,0)$ be a smooth map-germ: we denote by $\theta(f)$ the $\ofu_n$-module of vector fields along $f$, which can be identified with $\ofu_n\times\overset{p}\cdots\times\ofu_n$.
Recall that $\A = \text{Diff}_0(n)\times\text{Diff}_0(p)$ acts over the set of map-germs from $\bbk^n\to\bbk^p$ with composition in the source and target, being $\text{Diff}_0(s)$ the group of germs of diffeomorphisms $(\bbk^s,0)\to(\bbk^s,0)$.
Similarly, $\R = \text{Diff}_0(n)$ acts on map-germs via composition on the source.
Lastly, $\K$ is the subgroup of $\text{Diff}_0(n+p)$ given by diffeomorphisms of the form $\Phi(x,X) = (\phi(x),\psi(x,X))$ where $\phi\in\text{Diff}_0(n)$ and $\psi(x,0)=0$.
They act on each $\mononp f$ by assigning the map-germ $\psi\left(\phi^{-1}(x),f(\phi^{-1}(x)\right)$.

Associated to $f$ we can define the mappings $tf\colon\theta_n\to\theta(f)$, $\omega f\colon\theta_p\to\theta(f)$, given by $tf(\xi) = df(\xi)$ and $\omega f(\eta) = \eta\circ f$ for each $\xi\in\theta_n$, $\eta\in\theta_p$.
Then, we define the extended $\A$-tangent space and the extended $\A$-normal space as:
\begin{align*}
\tae f &= tf(\theta_n) + \omega f(\theta_p) \\
\nae f &= \frac{\theta(f)}{tf(\theta_n) + \omega f(\theta_p) }
\end{align*}
and the $\A_e$-codimension as $\aecod(f) = \dim_\bbk \nae f$.
We say that $f$ is $\A$-finite if $\aecod(f)<\infty$.
Similarly, for $\K$-equivalence we have:
$$\nke f =\frac{\theta(f)}{T\mathscr K_e f}= \frac{\theta(f)}{tf(\theta_n)+f^*\mfr_p \theta(f)}$$

\begin{definition}
A diagram of the form:
\begin{equation*}
\xymatrix{
(\bbk^{n_2},0) \ar[r]^F & (\bbk^{p_2},0)\\
(\bbk^{n_1},0) \ar[r]^f \ar[u]^j& (\bbk^{p_1},0)\ar[u]^i\\
}
\end{equation*}
is called a {\it transverse fibre square} if $i$ is transverse to $F$ and $(f,j)$ maps $(\bbk^{n_1},0)$ diffeomorphically into the germ at 0 of the submanifold of points $(y,x)\in \bbk^{p_1}\times\bbk^{n_2}$ such that $i(y)=F(x)$.
Here $f$ is called the transverse pull-back of $F$ by $i$ and is denoted $i^*F$.
\end{definition}

The following theorem is due to Damon, see \cite{damonakv}:

\begin{teo}\label{thm_damon}
Assume that we have a diagram:
\begin{equation*}
\xymatrix{
(\bbk^{n+s},0) \ar[r]^F & (\bbk^{p+s},0)\ar[r]^\pi&(\bbk^s,0)\\
(\bbk^{n},0) \ar[r]^f \ar[u]^j& (\bbk^{p},0)\ar[u]^i &\\
}
\end{equation*}
where the left square is a transverse fibre square and $\pi\circ i = 0$.
Then:
$$\nae f \cong \frac{\theta(\pi)}{ d\pi(\operatorname{Derlog}(\Delta F)) + \pi^*\mfr_s\theta(\pi)}$$
where $\operatorname{Derlog}(\Delta F)$ is the $\ofu_{p+1}$-module of vector fields tangent to the discriminant of $F$, $\Delta F$.
\end{teo}

\begin{definition}
Let $\mononp f$ be a smooth map-germ, a 1-parameter stable unfolding (OPSU) $F$ of $f$ is a map-germ $F\colon (\bbk^n\times\bbk,0)\to(\bbk^p\times\bbk,0)$ of the form $F(x,\lambda) = (f_\lambda(x),\lambda)$ with $f_0(x) = f(x)$ which is stable as a map-germ.
\end{definition}
We will often denote $\bar f(x,\lambda) = f_\lambda(x)$.
Then $d_x\bar{f}$ will denote the Jacobian matrix of $\bar f$ with respect to the first $n$ variables, and $d_\lambda\bar f$ the Jacobian matrix with respect to the parameter variable.
Note specifically that $d_\lambda\bar f \in \theta(\bar f)$.
We will denote the target variables of $F$ as $(X,\Lambda) = (X_1,\ldots,X_p,\Lambda)$, and the constant vector fields in $\bbk^p\times\bbk$ by $\dpar{}{X_1},\ldots,\dpar{}{X_p},\dpar{}{\Lambda}$.
For more information on 1-parameter stable unfoldings and their relevance, see \cite{notaopsus}.

\begin{definition}
Let $\mononp f$ be a smooth map-germ.
The {\it isosingular locus} of $f$ is the germ at 0 of the set of points $y\in\bbk^p$ such that $f\colon(\bbk^n,f^{-1}(y))\to(\bbk^p,y)$ is $\A$-equivalent to $f$.
This is a well defined germ of a submanifold in $\bbk^p$,  often called the {\it analytic stratum of $f$}, and its tangent space at $0$ is denoted by $\tilde\tau (f)$.
\end{definition}

See Chapter 7, section 7.2 in \cite{nunomond} for more information on the isosingular locus.
For the purposes of this paper, it is sufficient to know that, given $F$ a stable map-germ, the dimension of $\tilde{\tau}(f)$ is equal to the maximum $d\in\bbn$ such that $F\sim_\A H\times Id_d$, for $H$ a stable map-germ.

Recall that the Tjurina number of a function $g\in\ofu_d$ is defined as:
$$\tau(g) = \dim_\bbk\frac{\ofu_d}{Jg + g}$$
where $Jg + g$ is the ideal in $\ofu_d$ generated by the partial derivatives of $g$ and $g$.
We say that a function $g\in\ofu_d$ is quasi-homogeneous if there exist weights $w,w_1,\ldots,w_d\in\bbn$ such that for each $\lambda\in\bbk$:
$$g(\lambda^{w_1}z_1,\ldots,\lambda^{w_d}z_d) = \lambda^wg(z_1,\ldots,z_d)$$
Similarly, a map-germ $\mnp f \bbk$ will be quasi-homogeneous if there exist weights such that each component of $f$ is quasi-homogeneous with respect to those same weights.

\begin{definition}
Given $\mononp f$, a vector field $\eta\in\theta_p$ is said to be liftable if there exists $\xi\in\theta_n$ such that $tf(\xi) = \omega f (\eta)$.
The set of liftable vector fields of $f$ is denoted by $\Lift(f)$. The evaluation at 0 of vector fields in $\Lift(f)$ gives precisely $\tilde\tau (f)$ (see pages 281--283 in \cite{nunomond}).
In many cases, $\operatorname{Lift}(f) = \operatorname{Derlog}(\Delta f)$; in particular, this holds when $f$ is stable.
\end{definition}

If $\pi\colon(\bbk^{p}\times\bbk,0)\to(\bbk,0)$ is the canonical projection given by $\pi(X,\Lambda) = \Lambda$, then the differential $d\pi$ projects vector fields in $\theta_{p+1}$ to their last component.
In particular, if $M\subseteq \theta_{p+1}$ is an $\ofu_{p+1}$-module, $d\pi(M)$ is the ideal in $\ofu_{p+1}$ generated by all the functions in the last component of the vector fields in $M$.

\begin{definition}\label{def_degsus}
Let $F$ be an OPSU of $\mononp f$, we define the {\it degree of substantiality} of $F$ as the minimum $m\in\bbn$ such that $\Lambda^m\in d\pi(\Lift(F))$.
We denote it by $\delta(F) = m$.
If there is no such minimum, we write $\delta(F) = \infty$ 
\end{definition}

%

The last definition aims to generalize the concept of {\it substantial unfolding} defined by Houston in \cite{houstonaug2}: we say in particular that an OPSU is substantial if it has degree of substantiality equal to $1$.
Notice that any quasi-homogeneous OPSU is substantial, since in this case the Euler vector field $\sum_{i=1}^p w_iX_i\dpar{}{X_i} +w_{p+1}\Lambda\dpar{}{\Lambda}$ for some weights $w_1,\ldots,w_p,w_{p+1}\in\bbn$ is liftable.

In particular, if $f$ is quasi-homogeneous and admits an OPSU, we can find a vector field $\gamma\in\theta(f)$ so that $(f(x) +\lambda\gamma(x),\lambda)$ is stable and quasi-homogeneous.
This is due to the fact that this $\gamma$ will be the only non-constant $\bbk$-generator of the extended $\K$-normal space,
and so $\gamma$ in this space will always lie on the class of a vector field with all components equal to 0 except for one with a single monomial in $x_1,\ldots,x_n$. 
Given a representative $\gamma$ in this form, if $f$ is quasi-homogeneous, $(f(x) +\lambda\gamma(x),\lambda)$ will also be quasi-homogeneous.

\begin{rem}\label{rem_subs}
Let $F(x,\lambda) = (f_\lambda(x),\lambda)$ be a substantial unfolding of $f$ (i.e., $\delta(F) = 1$), then there exist $\eta(X,\Lambda) = (\eta_1(X,\Lambda),\ldots,\eta_p(X,\Lambda),\Lambda)\in\text{Lift}(F)$ and $\xi = (\xi_1,\ldots,\xi_{n+1})\in\theta_{n+1}$ such that $dF(\xi) = \eta\circ F$.
Computing $dF$, we see that this implies $\xi_{n+1}(x,\lambda) = \lambda$.

Then, by the Thom-Levine lemma (see Lemma 2.6 from \cite{nunomond} or Proposition 4.12 in \cite{ruas_resumen}), we can integrate these vector fields to find families of diffeomorphisms $\psi_t,\phi_t$ such that $\psi_t\circ F = F\circ\phi_t$ for all $t\in\bbk$ small enough, $\phi_0 = \id_{n+1},\psi_0 = \id_{p+1}$ and $\dpar{\psi_t}{t} = \eta\circ \psi_t, \dpar{\phi_t}{t} = \xi\circ\phi_t$.

In particular, if $\psi_t =(\psi_t^1,\ldots,\psi_t^{p+1})$, we can deduce from the last condition that $\dpar{\psi_t^{p+1}}{t} =\eta^{p+1}\circ\psi_t=\psi_t^{p+1}$, and so solving the differential equation we can consider $\psi_t^{p+1}(X,\Lambda) = \Lambda e^{t}$.
Similarly, we can take $\phi_t(x,\lambda) = (\phi_t^1(x,\lambda),\ldots,\phi_t^{n}(x,\lambda),\lambda e^t)$.
Moreover, for each $t\in\bbk$ small enough, the inverse of $\phi_t$ is also a diffeomorphism of the form $\phi_t^{-1} = (\phi_t^{-1,1},\ldots,\phi_t^{-1,n},\lambda e^{-t})$.

These diffeomorphisms, which maintain the parameter variable on the last component, will be of importance in later sections. In particular, in Proposition \ref{propo_weak_subs} we show that they preserve the property of an unfolding being substantial.
Diffeomorphisms $(\Phi,\Psi)\in\A$ such that $\Psi\circ F = F\circ \Phi$ form a subgroup called the {\it isotropy group} or {\it stabilizer}, and is often denoted by $\A_F$.
If $(\Phi,\Psi)\in\A$ are diffeomorphisms such that $G \circ \Phi = \Psi\circ F$, then $\A_G = (\Phi,\Psi)\cdot \A_F \cdot (\Phi^{-1},\Psi^{-1})$.

\end{rem}

We now present some known results on augmentations.
In \cite{augcod1morse} we proved that in some cases, the resulting augmentation is independent of the chosen OPSU.
\begin{prop}\label{prop_opsu}
Assume $\aecod f = 1$ and $g$ is $\R$ equivalent to a quasi-homogeneous function with isolated singularity.
Given any OPSUs $F$ and $F'$ of $f$:
$$A_{F,g}(f)\sim_\A A_{F',g}(f)$$
\end{prop}


Also in \cite{augcod1morse} we proved that, given a fixed OPSU $F$, the $\A$-class of the augmentation only depends on the $\R$-class of the augmenting function:
\begin{prop}\label{prop_requiv_augm_aequiv}
If $g$ is $\R$-equivalent to $g'$, then:
$$A_{F,g}(f) \sim_\A A_{F,g'}(f)$$
\end{prop}

To finish the section, and for the sake of completion, we recall the definition of $\A$-simple map-germ, which can be extended analogously to $\R$-simplicity (and $\K$-simplicity) of functions:

\begin{definition}
A smooth $f\colon(\bbk^n,0)\to(\bbk^p,0)$ is $\A$-simple if there exist a finite number of $\A$-equivalence classes such that, if the versal unfolding of $f$ admits a representative $F\colon U\to V $ with $U\subseteq \bbk^n\times\bbk^d, V\subseteq\bbk^p\times\bbk^d$, of the form $F(x,\lambda) = (f_\lambda(x),\lambda)$, for each $(y,\lambda)\in V$ and each $(x_1,\lambda),\ldots,(x_s,\lambda)\in U\times V$ with $F(x_i,\lambda) = (y,\lambda)$ the map-germ $f_\lambda\colon(\bbk^n,\lbrace x_1,\ldots,x_s\rbrace)\to(\bbk^p,y)$ lies in one of those classes.
\end{definition}

\section{Codimension and Image Milnor number}

For the purposes of this section, we will consider only the complex analytic case, $\bbk = \bbc$.
In \cite{houstonaug2}, Houston proved the following inequality:
$$\aecod(A_{F,g}(f))\geq \aecod(f)\tau(g)$$
with equality if $g$ is quasi-homogeneous or $F$ is substantial. 
In fact he introduced there the concept of {\it substantiality} and related it to this result.

The goal of this section is to find an upper bound for the codimension of an augmentation, which will help us tackle Conjectures 1 and 2 for the case of augmentations.
%
%

\begin{definition}
The Briançon-Skoda exponent of a smooth function $g\colon(\bbk^d,0)\to(\bbk,0)$ is defined as the minimum positive integer $r$ such that $g^r \in Jg$.
We denote it by $BS(g)=r$.
\end{definition}

The fact that the Briançon-Skoda number is well defined comes from \cite{briskoda}, where it is proven that $BS(g) \leq d$ for every $g\in\ofu_d$.

We now want to provide an upper bound to the codimension of an augmentation, and extend the conditions of Houston's theorem to provide an if and only if condition on the equality of the lower bound.
Our setting is the following: let $f\colon(\bbc^n,0)\to(\bbc^p,0)$ be a complex analytic map-germ that admits an OPSU $F(x,\lambda) = (f_\lambda(x),\lambda)$.
If $\pi\colon(\bbc^{p+1},0)\to(\bbc,0)$ is the projection $\pi(X,\Lambda)= \Lambda$, then $d\pi(\Lift(F))$ is the ideal in $\ofu_{p+1}$ generated by the last components of the liftable vector fields of $F$.
By Damon's Theorem \ref{thm_damon} we have that:
$$\aecod(f) = \dim_\bbc\frac{\ofu_{p+1}}{d\pi(\Lift(F))+\langle \Lambda\rangle}$$
where $\langle \Lambda \rangle$ is the ideal generated by $\Lambda$ in $\ofu_p+1$.
Hence, if $k=\aecod(f)$, for $i=1,\ldots,k$, we can find $h_i\in \ofu_{p+1}$ such that:
$$\frac{\ofu_{p+1}}{d\pi(\Lift(F))+\langle \Lambda\rangle} = \linsp{\bbc}{[h_1(X,\Lambda)],\ldots,[h_k(X,\Lambda)]}$$
where $[h_i]$ are the classes in the quotient of each $h_i$.
Moreover by Hadamard's Lemma we can find for each $i$ some $\tilde{h_i}\in\ofu_{p+1}$ such that $h_i(X,\Lambda) = h_i(X,0) + \Lambda \tilde{h_i}(X,\Lambda)$.
Since $\Lambda\tilde{h_i} \in \langle \Lambda\rangle$, we have that $[h_i(X,\Lambda)] = [h_i(X,0)]$.
That is, we can take a basis of this quotient with representatives that do not depend on $\Lambda$.
We will write this basis from now on as $h_1(X),\ldots,h_k(X)$.

Consider $g\colon(\bbk^d,0)\to(\bbk,0)$ a smooth function such that:
$$\mu(g) = \dim_\bbc\frac{\ofu_d}{Jg} < \infty$$
where $Jg$ is the ideal in $\ofu_d$ generated by the partial derivatives of $g$.
The number $\mu(g)$ is called the Milnor number of $g$, and it is finite precisely when $g$ is finitely $\R$-determined.
There exist some $m_1,\ldots,m_r\in\ofu_d$  with $r= \mu(g)$ such that:
$$\frac{\ofu_d}{Jg} = \linsp{\bbc}{m_1(Z),\ldots,m_r(Z)}$$
Then:
$$\frac{\ofu_{p+1}}{d\pi(\Lift(F))+\langle \Lambda\rangle}\otimes_\bbc\frac{\ofu_d}{Jg} = \linsp{\bbc}{h_i(X)\otimes_\bbc m_j(Z)}_{\substack{i=1,\ldots,k\\ j=1,\ldots,r}}$$
In the analytic case:
\begin{equation*}
\frac{\ofu_{p+1+d}}{d\pi(\Lift(F))+\langle \Lambda\rangle + Jg} \cong \frac{\ofu_{p+1}}{d\pi(\Lift(F))+\langle \Lambda\rangle}\otimes_\bbc\frac{\ofu_d}{Jg}
\end{equation*}
Hence, via the isomorphism that multiplies monomials:
\begin{equation}\label{eq_decomp}
\frac{\ofu_{p+1+d}}{d\pi(\Lift(F))+\langle \Lambda\rangle + Jg} = \linsp{\bbc}{h_i(X)m_j(Z)}_{\substack{i=1,\ldots,k\\ j=1,\ldots,r}}
\end{equation}

Now, Houston proved in \cite{houstonaug2} that the $\A_e$ codimension of the augmentation of $f$ by $F$ via $g$ is given by the dimension as a $\bbc$-vector space of the quotient:
\begin{equation}\label{eq_ideal_aecod}
B = \frac{\ofu_{p+1+d}}{d\pi(\Lift(F))+\langle \Lambda-g\rangle + Jg}
\end{equation}
where $\langle \Lambda-g\rangle$ is the ideal generated by the function $\Lambda-g(Z)$.
We will now see that this dimension is leser than or equal to $kr = \aecod(f)\mu(g)$.
For $\alpha\in\ofu_{p+d+1}$, denote by $[\alpha]_B$ the class of $\alpha$ in $B$.
We will proceed with all this setting and notation for the following two results.
\begin{lem}\label{lemma_prep_class}
For each $\alpha \in \ofu_{p+d+1}$, there exist unique $\alpha_{ij}\in\bbc$ and some $b\in\ofu_{p+d+1}$ such that:
$$[\alpha]_B = \left[\sum_{i=1}^k\sum_{j=1}^r h_i(X)m_j(Z)\alpha_{ij} + g(Z)b(X,Z,\Lambda)\right]_B$$
\end{lem}
\begin{proof}
By Equation \ref{eq_decomp}, we can find unique $\alpha_{ij}\in\bbc$ and some $\eta\in d\pi(\Lift(f)), b\in\ofu_{p+d+1}$ and $c\in Jg$ such that we can decompose $\alpha$ as:
$$\alpha(X,Z,\Lambda) = \sum_{i=1}^k\sum_{j=1}^r h_i(X)m_j(Z)\alpha_{ij} + \eta(X,Z,\Lambda) + \Lambda b(X,Z,\Lambda) + c(X,Z,\Lambda)$$
Now, add and substract $g(Z)b(X,Z,\Lambda)$ on the right-hand side to get:
\begin{multline*}
\alpha(X,Z,\Lambda) = \sum_{i=1}^k\sum_{j=1}^r h_i(X)m_j(Z)\alpha_{ij} + g(Z)b(X,Z,\Lambda) \\
+ \eta(X,Z,\Lambda) + (\Lambda - g(Z)) b(X,Z,\Lambda) + c(X,Z,\Lambda)
\end{multline*}
Taking classes in $B$ we obtain the desired result.
\end{proof}

\begin{teo}\label{thm_ineq_iff}

Let $ f\colon(\bbc^n,0)\to(\bbc^p,0)$ be a smooth map-germ with an OPSU $F(x,\lambda) = (f_\lambda(x),\lambda)$ and $g\colon(\bbc^d,0)\to(\bbc,0)$ a smooth function with isolated singularity.
Then:
$$\aecod(f)\tau(g)\leq \aecod(A_{F,g}(f))\leq \aecod(f)\mu(g)$$
Moreover, the lower bound is an equality if and only if $g$ is quasi-homogeneous or $F$ is substantial. If $g$ is quasi-homogeneous, the upper bound is also an equality.
\end{teo}

\begin{proof}

The lower bound is given in \cite{houstonaug2}, along with the {\it if} condition of the equality.
For the {\it only if} condition, in the cited proof it is seen that:
$$\aecod (A_{F,g}(f)) = \dim_\bbc\frac{\ofu_{p+1+d}}{d\pi(\Lift(F)) + \langle \Lambda- g(Z)\rangle + Jg}$$
Notice that the vector fields in $\Lift(F)$ only depend on $X,\Lambda$, and the functions in $Jg$ only depend on $Z$.
The following inclusion holds:
$$d\pi(\Lift(F)) + \langle \Lambda- g(Z)\rangle + Jg \subseteq d\pi(\Lift(F)) + \langle \Lambda, g(Z)\rangle + Jg$$
The ideal in the right-hand side of the inclusion has codimension over $\ofu_{p+1+d}$ as a $\bbc$-vector space equal to $\aecod(f)\tau(g)$.

If we assume $\aecod(A_{F,g}(f)) = \aecod(f)\tau(g)$ then both ideals have the same codimension, and since one is contained in the other they must be equal.
Since $\Lambda$ cannot be found in the ideal generated by $Jg$ and $g$, either $g$ is in $Jg$, or $\Lambda\in d\Lambda(\Lift (F))$.
In the former case, Saito proved in \cite{saito} that this means $g$ is ($\R$-equivalent to a) quasi-homogeneous function; in the latter, it means $F$ is substantial.

To obtain the upper bound, we will see that the classes of $h_i(X)m_j(Z)$ also span $B$ over $\bbc$.
Let $\alpha\in\ofu_{p+d+1}$.
By Lemma \ref{lemma_prep_class} there exist $b_1\in\ofu_{p+d+1}$ and some $\alpha_{ij}^0\in\bbc$ for $i=1,\ldots,k$ and $j=1,\ldots,k$ such that:
$$[\alpha]_B = \left[\sum_{i=1}^k\sum_{j=1}^r h_i(X)m_j(Z)\alpha^0_{ij} + g(Z)b_1(X,Z,\Lambda)\right]_B$$
We can apply Lemma \ref{lemma_prep_class} again to find $\alpha_{ij}^1\in\bbc$ and $b_{2}\in\ofu_{p+d+1}$ such that:
$$[b_1]_B = \left[\sum_{i=1}^k\sum_{j=1}^r h_i(X)m_j(Z)\alpha^1_{ij} + g(Z)b_{2}(X,Z,\Lambda)\right]_B$$
Denote by $q = \min\lbrace BS(g), \delta(F)\rbrace$.
Proceeding by finite induction, we can repeatedly apply the same lemma to obtain some $b_1,\ldots,b_{q}\in\ofu_{p+d+1}$ and some $\alpha_{ij}^l\in\bbc$ such that for each $l=1,\ldots,q-1$:
$$[b_l]_B = \left[\sum_{i=1}^k\sum_{j=1}^r h_i(X)m_j(Z)\alpha^l_{ij} + g(Z)b_{l+1}(X,Z,\Lambda)\right]_B$$

Hence:
\begin{align*}
[\alpha]_B & = \left[\sum_{i=1}^k\sum_{j=1}^r h_i(X)m_j(Z)\sum_{l=0}^{q-1} g(Z)^l\alpha_{ij}^l + g(Z)^{q}b_{q}(X,Z,\Lambda)\right]_B \\
 & = \left[\sum_{i=1}^k\sum_{j=1}^r h_i(X)m_j(Z)\sum_{l=0}^{q-1} g(Z)^l\alpha_{ij}^l\right]_B
\end{align*}
Where the last equality comes from the fact that $g(Z)^{q}$ belongs to the ideal $d\pi(\Lift(F)) + \langle \Lambda - g\rangle + Jg$.
This can be seen like this: if $q = BS(g)$, then $g(Z)^q\in Jg$.
If $q = \delta(F)$, this implies $\Lambda^q \in d\pi(\Lift(f))$, and we have:
$$g(Z)^q = \Lambda^q - \sum_{i= 1}^q \Lambda^{q-i}g(Z)^{i-1}(\Lambda-g(Z)) \in d\pi(\Lift(F)) + \langle\Lambda-g\rangle$$

Now, for each $j = 1,\ldots,r$ and $l=0,\ldots, q-1$ we have $m_j(Z)g(Z)^l\in\ofu_d$ and so we can find some $\beta_{s}^{jl}\in\bbc$ and a function $C_{jl} \in Jg$ such that:
$$m_j(Z)g(Z)^l = \sum_{s=1}^r m_s(Z)\beta_{s}^{jl} + C_{jl}(Z)$$
That is:
\begin{multline*}
[\alpha]_B = \left[\sum_{i=1}^k\sum_{j=1}^r \sum_{l=0}^{q-1} h_i(X)m_j(Z)g(Z)^l\alpha_{ij}^l\right]_B = \\
\left[\sum_{i=1}^k\sum_{j=1}^r \sum_{l=0}^{q-1} h_i(X)\alpha_{ij}^l\left(\sum_{s= 1}^r m_s(Z)\beta_{s}^{jl} + C_{jl}(Z)\right)\right]_B
\end{multline*}
Now, using that $C_{jl}\in Jg$ for all $j,l$ we have that:
\begin{align*}
[\alpha]_B & = \left[\sum_{i=1}^k\sum_{j=1}^r \sum_{l=0}^{q-1} h_i(X)\alpha_{ij}^l\sum_{s= 1}^r m_s(Z)\beta_{s}^{jl} \right]_B\\
& = \left[\sum_{i=1}^k\sum_{s= 1}^r h_i(X) m_s(Z)\left(\sum_{j=1}^r \sum_{l=0}^{q-1}  \alpha_{ij}^l\beta_{s}^{jl}\right) \right]_B
\end{align*}
Therefore, the class of every $\alpha$ lies in the class of some $\bbc$-linear combination of the $h_i(X),m_s(Z)$.

\end{proof}

\begin{rem}
The {\it{only if}} part for equality  with the lower bound is important in the sense that if we find a germ $f$ which admits an OPSU $F$ which is substantial and another OPSU $F'$ which is not and we augment by a non-quasi-homogeneous function $g$, then $\aecod (A_{F,g}(f))<\aecod (A_{F',g}(f))$ and so the $\mathscr A$-class of the augmentation would depend on the choice of unfolding. In Sections 5 to 8 we study sufficient conditions to ensure the independence of the unfolding.
\end{rem}

Although the use of $\delta(F)$ seems redundant in the proof (since the Briançon-Skoda number is always finite) we have chosen to use it since it seems like both numbers play a role in determining the $\A_e$-codimension of the augmentation.
Also, while the lower bound is sharp, the upper bound is not (unless $g$ is quasi-homogeneous), as the following example shows:

\begin{ex}
Consider the following non-quasihomogeneous map-germs, which can be found in \cite{rieger2to2}, \cite{marartari} and \cite{houstonkirk} respectively.
\begin{align*}
11_5  &\equiv (x,y^4 + xy^2 + x^2y)\\
5_2   &\equiv (x,y,z^5 + xz +y^2z^2 +yz^3)\\
P_3^2 &\equiv (x,y,yz+z^6 + z^8,xz+z^3)
\end{align*}
The first two have $\A_e$ codimension 2, while the third has $\A_e$-codimension 3.
All 3 have an OPSU $F$ by adding $(0,ty,t),(0,0,tz^2,t)$ and $(0,0,tz^3,0,t)$, respectively. Using SINGULAR one can check that all the OPSUs have degree of substantiality 2.
Consider also the following non-quasihomogeneous functions:
\begin{align*}
DG_k(u,v) & = u^{2k+1} + u^kv^{k+1} + v^{2k}, \; k \geq 3 \\
M(u,v,w) & = (uvw)^2 +u^8+v^8+w^8
\end{align*}
The first family was defined by Dimca and Greuel in \cite{dimcagreuel}, where they showed that the quotient ${\mu}/{\tau}$ for the members of this family tends to $4/3$ as $k$ tends to infinity.
Here $BS(DG_k) = 2$ for each $k\geq 3$.
The second function was defined by Malgrange in \cite{malgrange_letter}, and has the property that $BS(M) = 3$, which is the maximum possible for a $3$-variable function.
We have that $\tau(DG_k) = 3k^2, \mu(DG_k) = 2k(2k-1)$ and $\tau(M) = 179,\mu(M) = 215$.
Table \ref{table_codims} shows the computations, obtained by computing the codimension of the ideal from Equation \ref{eq_ideal_aecod} in each case via SINGULAR.
We also show the lower and upper bounds, and an estimation that seems to hold in the cases where the augmented map-germ is primitive.

\begin{table}\label{table_codims}
\centering
\hspace*{-1.5cm}
\begin{tabular}{ccccccc}
\hline
$A_{F,g}(f)$     & $cod(f)\tau(g)$ & $cod(A_{F,g}(f))$ & $cod(f)\mu(g)$ & $cod(f)\tau(g)+\mu(g)-\tau(g)$ &   \\
\hline
$A_{F,DG_3}(11_5)$ & 54                                 & 57                       & 60                                & 57                                                                                  &   \\
$A_{F,DG_4}(11_5)$ & 96                                 & 104                      & 112                               & 104                                                                                 &   \\
$A_{F,DG_5}(11_5)$ & 150                                & 165                      & 180                               & 165                                                                                 &   \\
$A_{F,M}(11_5)$    & 358                                & 393                      & 430                               & 394                                                                                 &   \\
$A_{F,DG_3}(5_2)$ & 54                                 & 57                       & 60                                & 57                                                                                  &   \\
$A_{F,DG_4}(5_2)$ & 96                                 & 104                      & 112                               & 104                                                                                 &   \\
$A_{F,DG_5}(5_2)$& 150                                & 165                      & 180                               & 165                                                                                 &   \\
$A_{F,M}(5_2)$ & 358                                & 393                      & 430                               & 394                                                                                 &   \\
$A_{F,DG_3}(P_3^2)$ & 81                                 & 84                       & 90                                & 84                                                                                  &   \\
$A_{F,DG_4}(P_3^2)$ & 144                                & 152                      & 168                               & 152                                                                                 &   \\
$A_{F,DG_5}(P_3^2)$ & 225                                & 240                      & 270                               & 240                                                                                 &   \\
$A_{F,DG_M}(P_3^2)$     & 537                                & 572                      & 645                               & 573 &   \\
$A_{F,DG_3}(A_{F,DG_3}(11_5))$ & 1539                               & 1629                     & 1710                              & 1542                                                                                &   \\
$A_{F,M}(A_{F,M}(11_5))$      & 70347                              & 77908                    & 84495                             & 70383                                                                               &
\end{tabular}
\medskip
\caption{Codimensions and bounds of different augmentations}
\end{table}
\end{ex}

\begin{rem}
The previous example and further extensive study of examples suggests that a smaller upper bound is given by the last column in Table \ref{table_codims}, i.e. $\aecod(A_{F,g}(f))\leq \aecod(f)\tau(g) + \mu(g) - \tau(g)$, with equality if and only if $\delta(F)\geq BS(g)$. This however is not true for the last two examples in the table. The reason seems to be that these examples are augmentations of augmentations and for that upper bound to hold one must consider the Milnor and Tjurina numbers of the join of all the augmenting functions that lead to this germ. This makes things much more complicated and we have not been able to prove this. However, the Milnor and Tjurina numbers of the join of $DG_3$ and $DG_3$ are $900$ and $738$ respectively, and of the join of $M$ with $M$ are $46225$ and $33267$ respectively. Here $\aecod(f)\tau(g) + \mu(g) - \tau(g)$ is equal to $1638$ for the first case and $79492$ for the second, so the conjecture still holds.

\end{rem}

We now comment on the Mond conjecture for augmentations. The Mond conjecture for a map-germ $f:(\mathbb C^n,0)\to(\mathbb C^{n+1},0)$ states that $\aecod(f)\leq \mu_I(f)$ with equality if $f$ is quasi-homogeneous.
Here $\mu_I(f)$ is the image Milnor number, an analytic invariant which counts the number of spheres in the image of a stable perturbation of $f$. In Theorem 6.7 in \cite{houstoninventiones} Houston proves that if a corank 1 map-germ $f$ satisfies Mond's conjecture and $g$ defines an isolated hypersurface singularity then, given that $f$ or $g$ are quasi-homogeneous, $\aecod(f)\tau(g)=\aecod(A_{F,g}(f))\leq \mu_I(A_{F,g}(f))=\mu_I(f)\mu(g)$ with equality if both $f$ and $g$ are quasi-homogeneous (where $\mu(g)$ is the Milnor number of $g$). In the proof he uses his Theorem 3.3 in \cite{houstonaug}, which is a previous version of Theorem 4.4 in \cite{houstonaug2}, so, using this second theorem in the proof instead of the first one, the condition for the inequality of $f$ being quasi-homogeneous can be substituted for $F$ being substantial.

On the other hand, in Corollary 2.17 of \cite{gimeneznunoweak}, Giménez Conejero and Nu\~no-Ballesteros note that Houston's equality $\mu_I(A_{F,g}(f))=\mu_I(f)\mu(g)$ is true for any corank. Using their Corollary 2.16 which states that any $\mathscr A_e$-codimension 1 germ of any corank satisfies the Mond conjecture, they conclude that any augmentation of any $\mathscr A_e$-codimension 1 germ satisfies the Mond conjecture.

Using these results and our Theorem \ref{thm_ineq_iff} we can give a much more general result.

\begin{teo}\label{mondconj}
Let $f:(\mathbb C^n,0)\to(\mathbb C^{n+1},0)$ (of any corank) admit an OPSU $F$ and suppose that $f$ satisfies Mond's conjecture. Let $g$ define an isolated hypersurface singularity. Then $$\aecod(A_{F,g}(f))\leq \mu_I(A_{F,g}(f)),$$ with equality if $f$ and $g$ are quasi-homogeneous.
\end{teo}

\begin{proof}
We have
\begin{align*}
\aecod(A_{F,g}(f))&\leq \aecod(f)\mu(g) \text{, by Theorem \ref{thm_ineq_iff},}\\
&\leq \mu_I(f)\mu(g) \text{,  by Mond's conjecture for $f$,}\\
&=\mu_I(A_{F,g}(f)) \text{, by Corollary 6.4 in \cite{houstoninventiones} and Corollary 2.17 in \cite{gimeneznunoweak}.}
\end{align*}
Now, if $g$ is quasi-homogeneous, the first inequality turns to an equality, and if $f$ is quasi-homogeneous, the second inequality turns to an equality.
\end{proof}


Notice also that by construction, if $f$ is quasi-homogeneous it admits a quasi-homogeneous OPSU $F(x,\lambda)=(f(x)+\lambda\gamma(x),\lambda)$, and if $g$ is also quasi-homogeneous then $A_{F,g}(f)$ is quasi-homogeneous. 
We can also prove some partial converses for the equality in Theorem \ref{mondconj}.

\begin{coro}
Let $f$ admit an OPSU $F(x,\lambda)=(f(x)+\lambda\gamma(x),\lambda)$, suppose it satisfies the Mond conjecture and $\aecod(A_{F,g}(f))=\mu_I(A_{F,g}(f)).$ Then
\begin{enumerate}
\item[i)] $\aecod(f)=\mu_I(f)$.
\item[ii)] If $f$ is quasi-homogeneous, then $g$ is $\mathscr R$-equivalent to a quasi-homogeneous function.
\end{enumerate}
\end{coro}
\begin{proof}
For i) by Theorem \ref{thm_ineq_iff} $\aecod(A_{F,g}(f))\leq\aecod(f)\mu(g)$, so we get $\aecod(f)\mu(g)=\mu_I(f)\mu(g)$, so we get $\aecod(f)=\mu_I(f)$.

For ii) $f$ quasi-homogeneous implies that $F$ is substantial, so again by Theorem \ref{thm_ineq_iff} $\aecod(f)\tau(g)=\aecod(A_{F,g}(f))$, and since the Mond conjecture with $f$ quasi-homogeneous implies $\aecod(f)=\mu_I(f)$, we get $\tau(g)=\mu(g)$, which by Saito implies that $g$ is $\mathscr R$-equivalent to a quasi-homogeneous function.
\end{proof}

Next we turn our attention to Conjecture 2.

\begin{prop}\label{conj2n=1}
Let $f:(\mathbb C,0)\to(\mathbb C^{2},0)$, then $$\frac{\mu_I(f)}{\aecod(f)}\leq 2,$$ i.e. all parametrized plane curves satisfy Conjecture 2.
\end{prop}
\begin{proof}
All $f:(\mathbb C,0)\to(\mathbb C^{2},0)$ admits $g:(\mathbb C^2,0)\to(\mathbb C,0)$ such that $g\circ f=0$, and vice-versa, for all plane curve $g$ there is a parametrisation $f$. Now, by \cite{mondbentwires}, $\mu_I(f)=\delta-r+1$ where $\delta$ is the delta invariant of $g$ and $r$ is the number of branches. Since $\mu(g)=2\delta-r+1$ we get $\mu_I(f)=\mu(g)-\delta$. On the other hand in Proposition II.2.30(5) of \cite{greuellossenshustin}, it is proved that $\aecod(f)=\tau(g)-\delta$. Now, since in our case $r=0$, we have $\delta=\mu/2$, so we get
$$\frac{\mu_I(f)}{\aecod(f)}=\frac{\mu(g)-\delta}{\tau(g)-\delta}=\frac{\frac{\mu(g)}{2}}{\tau(g)-\frac{\mu(g)}{2}}=\frac{\frac{\mu(g)}{\tau(g)}}{2-\frac{\mu(g)}{\tau(g)}}.$$
By \cite{almiron1} $\frac{\mu(g)}{\tau(g)} < \frac{4}{3}$, so we get $$\frac{\mu_I(f)}{\aecod(f)}< 2.$$
\end{proof}

Next, notice that Conjecture 2 is trivially satisfied for all augmentations from $(\mathbb C^2,0)$ to $(\mathbb C^{3},0)$. In fact, the only parametrized plane curves which admit an OPSU are those $\mathscr A$-equivalent to $f_k=(y^2,y^{2k+1})$ for some $k\geq 1$. The OPSU is given by $F_k=(y^2,y^{2k+1}+\lambda y,\lambda)$. All of the $f_k$ are quasi-homogeneous so $\aecod(f_k)=\mu_I(f_k)$. Now, in order to have an augmentation with $n=2$, the augmenting function $g$ must have one variable only, and hence $\mathscr R$-equivalent to a quasi-homogeneous function (an $A_k$-singularity), which implies that $\mu(g)=\tau(g)$. Therefore $$\frac{\mu_I(A_{F_k,g}(f_k))}{\aecod(A_{F_k,g}(f_k))}=\frac{\mu_I(f_k)\mu(g)}{\aecod(f_k)\tau(g)}=1.$$ In conclusion

\begin{prop}
Let $h:(\mathbb C^2,0)\to(\mathbb C^{3},0)$ be an augmentation. Then $h$ is $\mathscr A$-equivalent to $A_{F_k,g}(f_k)$ and so $$\frac{\mu_I(h)}{\aecod(h)}=1\leq 3,$$ i.e. $h$ satisfies Conjecture 2.
\end{prop}

In general we can say the following

\begin{teo}\label{conj2aug}
Suppose that Conjecture 2 is satisfied for all map germs $f:(\mathbb C^m,0)\to(\mathbb C^{m+1},0)$ and for all $m<n$. Let $h:(\mathbb C^n,0)\to(\mathbb C^{n+1},0)$ be an augmentation. Then $$\frac{\mu_I(h)}{\aecod(h)}\leq\frac{1}{4}(n+1)^2.$$
\end{teo}
\begin{proof}
If $h$ is an augmentation, then $h=A_{F,g}(f)$ for a germ $f:(\mathbb C^k,0)\to(\mathbb C^{k+1},0)$ with OPSU F and an augmenting function $g:(\mathbb C^{n-k},0)\to(\mathbb C,0)$. By Theorem \ref{thm_ineq_iff} $\aecod(f)\tau(g)\leq\aecod(A_{F,g}(f))$. By Conjecture 2 for $f$ $\frac{\mu_I(f)}{\aecod(f)}\leq k+1$, and by \cite{liu_milnortjurina} $\frac{\mu(g)}{\tau(g)}\leq n-k$, so we have $$\frac{\mu_I(A_{F,g}(f))}{\aecod(A_{F,g}(f))}=\frac{\mu_I(f_k)\mu(g)}{\aecod(f_k)\tau(g)}\leq (k+1)(n-k).$$ Since $1\leq k\leq n-1$, the maximum of the function $(k+1)(n-k)$ is achieved when $k=(n-1)/2$ and substituting we get $\frac{1}{4}(n+1)^2$.
\end{proof}

\begin{coro}
If Conjecture 2 is true for all map germs with $n=2$, then it is true for all augmentations with $n=3$.
\end{coro}
\begin{proof}
By Proposition \ref{conj2n=1}, Conjecture 2 is true for $n=1$. Together with the hypothesis we get that Conjecture 2 is satisfied for all $n<3$, therefore, by Theorem \ref{conj2aug}, for all augmentations in $n=3$, the quotient will be less than $\frac{1}{4}(3+1)^2=4$.
\end{proof}

We can refine Theorem \ref{conj2aug} for the cases of augmenting functions $g\colon(\bbk^i,0)\to(\bbk,0)$ with $ri=2,3$.
When $i=2$, it was proven in \cite{almiron1} that $\frac{\mu(g)}{\tau(g)} <\frac{4}{3}$ is a sharp upper bound, in the sense that one can find functions whose quotient is as close to $\frac{4}{3}$ as desired.
In the case $i=3$ it is conjectured that $\frac{\mu(g)}{\tau(g)} < \frac{3}{2}$ is an upper bound: in \cite{almiron2} it is seen that this is true if Durfee's conjecture holds (see the mentioned reference for more information).

\begin{prop}
Suppose that $f:(\mathbb C^{n-i},0)\to(\mathbb C^{n-i+1},0)$ admits an OPSU 
$F$ and satisfies Conjecture 2. Let $g:(\mathbb C^{i},0)\to(\mathbb C,0)$ be a 
function germ, $i=2,3$. Then
\begin{enumerate}
\item If $i=2$, $A_{F,g}(f):(\mathbb C^{n},0)\to(\mathbb C^{n+1},0)$ satisfies 
Conjecture 2 for all $n<7$.
\item If $i=3$, and Durfee's conjecture holds, $A_{F,g}(f):(\mathbb C^{n},0)\to(\mathbb C^{n+1},0)$ satisfies 
Conjecture 2 for all $n<8$.
\end{enumerate}
\end{prop} 
\begin{proof}
For $i=2$, by \cite{almiron1}, $\frac{\mu(g)}{\tau(g)}<\frac{4}{3}$, so 
$$\frac{\mu_I(A_{F,g}(f))}{\aecod(A_{F,g}(f))}=\frac{\mu_I(f)\mu(g)}{\aecod(f)\tau(g)
}\leq \frac{4}{3}(n-1).$$ Now $\frac{4}{3}(n-1)<n+1$ if $n<7$.

Similarly, for $i=3$, by \cite{almiron2}, $\frac{\mu(g)}{\tau(g)}<
\frac{3}{2}$, so $$\frac{\mu_I(A_{F,g}(f))}{\aecod(A_{F,g}(f))}=\frac{\mu_I(f)\mu(g)}{\aecod(
f)\tau(g)}\leq \frac{3}{2}(n-2).$$ Now $\frac{3}{2}(n-2)<n+1$ if $n<8$.
\end{proof}

\section{Characterization of agumentations}

In \cite{houstonaug2}, Theorem 4.6, Houston claimed the following: if $f$ admits an OPSU $F$, then:
$$f \text{ is an augmentation } \iff\dim_\bbk\tilde{\tau}(F) \geq 1$$

His proof relies on the following statement, which appears as Lemma 2.6 in \cite{houstonaug2}:
suppose that $f\colon(\bbk^n,0)\to(\bbk^q,0)$ with $n\geq q$ is a smooth mono-germ of the form $f = \pi\circ i$, with $i\colon(\bbk^n,0)\to(\bbk^{n+r},0)$, $r\geq 0$ an immersion and $\pi\colon(\bbk^{n+r},0)\to(\bbk^q,0)$ a submersion. Then, $\dim_\bbk\tilde\tau(f)\geq q-r$.

This statement is proven false by the following counterexample: let $i\colon(\bbk^3,0)\to(\bbk^4,0)$, $\pi\colon(\bbk^4,0)\to(\bbk^3,0)$ and $\psi\colon(\bbk^4,0)\to(\bbk^4,0)$ be given by:
\begin{align*}
i(X,Y,Z) &= (X,Y,Z,0)\\
\psi(X,Y,Z,\Lambda) &= (\Lambda + X^3 + XY, Y, Z, X)\\
\pi(X,Y,Z,\Lambda) & = (X,Y,Z)
\end{align*}
Then, $\pi\circ\psi$ is a submersion and:
$$\pi\circ\psi\circ i(X,Y,Z) = (X^3 + XY,Y,Z)$$
This composition, which plays the role of $f$ in the statement, is a cuspidal edge, that is:
$$\dim_\bbk\tilde\tau(\pi\circ\psi\circ i) = 1 < 3-1  $$
and so it contradicts Houston's lemma.

This counterexample appears when trying to reproduce the proof of Theorem 4.6 in \cite{houstonaug2} to determine if the map-germ $C_3(x,y) = (x,y^2,xy^3 + x^3 y)$, which appears in Mond's classification of simple map-germs of $\bbk^2\to\bbk^3$ (\cite{mond2en3}), is an augmentation.
It admits the OPSU:
$$\tilde C_3(x,y,\lambda) = (x,y^2, xy^3 + x^3y + \lambda y, \lambda)$$
which has $\tilde\tau (F) = 1$.
We will see that, in fact, $C_3$ is not an augmentation and is thus a counterexample to Houston's theorem.

First we need a lemma which allows us to calculate what Tjurina number the augmenting function must have. Notice that given an augmentation $A_{F,g}(f)(x,z) = (f_{g(z)}(x),z)$ there is a natural OPSU given by $(f_{g(z)+\lambda}(x),z,\lambda)$.

\begin{lem}\label{lemma_tangencia}
Assume $a\colon(\bbk^n,0)\to(\bbk^p,0)$ is an augmentation and let $A\colon(\bbk^{n+1},0)\to(\bbk^{p+1},0)$ be the natural OPSU.
Then, the order of tangency at the origin of the natural immersion $i\colon(\bbk^p,0)\to(\bbk^{p+1},0)$ with the isosingular locus of $A$, is equal to the Tjurina number of the augmenting function.
\end{lem}
\begin{proof}
Since $a$ is an augmentation, there must exist a smooth function $g\colon(\bbk^d,0)\to(\bbk,0)$ and a smooth map-germ $f\colon(\bbk^{n-d},0)\to(\bbk^{p-d},0)$ that admits an OPSU $F(x,\lambda) = (f_\lambda(x),\lambda)$ such that:
$$a\sim_{\A} A_{F,g}(f)(x,z) = (f_{g(z)}(x),z)$$
Then the following is an OPSU for the augmentation:
$$\tilde F(x,z,\lambda) = (f_{g(z)+\lambda}(x),z,\lambda)$$
which is equivalent after a change in the source to:
$$(f_{\lambda}(x),z,\lambda - g(z))$$
and therefore is trivial along the germ at 0 of the manifold $M = \lbrace (0,z,-g(z)) \in \bbk^{p+1} : z\in\bbk^d\rbrace$.
The image of the inclusion $i$ is the plane $\lbrace(s,t,0)\in\bbk^{p+1} : (s,t)\in\bbk^{p-d}\times\bbk^d\rbrace$, which has order of tangency $\tau(g)$ with $M$.
\end{proof}

\begin{prop}
The map-germ $C_3(x,y) = (x,y^2, xy^3 + x^3y)$ is not an augmentation.
\end{prop}
\begin{proof}
Assume $C_3$ is an augmentation, that is, assume there exist some base map-germ $f\colon(\bbk,0)\to(\bbk^2,0)$, an OPSU $F(x,\lambda) = (f_\lambda(x),\lambda)$ of $f$, and a smooth function $g\colon(\bbk,0)\to(\bbk,0)$ such that $C_3\sim_\mathcal{A} A_{F,g}(f)$.

By Theorem \ref{thm_ineq_iff}:
$$3 = \aecod({C_3}) = \aecod\left({A_{F,g}(f)}\right) \geq \aecod (f)\tau(g)$$
This implies $3\geq \tau(g)$, which can only be true if $g$ is a simple function.
Since every simple function is $\R$-equivalent to a quasi-homogeneous function, we can apply Theorem \ref{thm_ineq_iff} again to see that:
$$3 = \aecod{\left(A_{F,g}(f)\right)} = \aecod (f)\tau(g)$$
This leaves us with two possibilities: either $\aecod (f) = 1$, or $\aecod (f) = 3$.
Assume first that $\aecod (f) = 1$, then $f$ is equivalent to $(y^2,y^3)$.
Also, $\tau(g) = 3$, which means that by Proposition \ref{prop_requiv_augm_aequiv} we only need to check the augmentation via $g(x) = x^4$.
By Proposition \ref{prop_opsu}, the class of the augmentation does not depend on the chosen OPSU.
Take for instance $F(\lambda, y) = (\lambda,y^2,y^3+\lambda y)$, which is stable, then:
$$A_{F,g}(f)(x,y) = (x,y^2,y^3+x^4y)$$
which lies in the equivalence class of $S_3$ as seen in \cite{mond2en3}.

Therefore, the only possiblity that remains is that $\aecod (f) = 3$ and $\tau(g) = 1$, which implies that $C_3$ must be the augmentation of $(y^2,y^7)$ via a quadratic function, $g(x) = x^2$.
Notice that augmenting by the OPSU $(y^2,y^7+\lambda y,\lambda)$ and via $x^2$ would get us $B_3$ in Mond's list \cite{mond2en3}, but since Proposition \ref{prop_opsu} does not apply here, we cannot be certain that augmenting via another OPSU won't give us a different map-germ (see Example \ref{ex_cusps_cross} which elaborates on this discussion).

In any case, $C_3$ admitted the OPSU:
$$C(x,y,\lambda) = (x,y^2,x^3y+xy^3+y\lambda,\lambda)$$
which is equivalent after a change in source to:
$$(x,y^2,\lambda y + xy^3,\lambda-x^3)$$
and therefore is trivial along the (germ at 0 of the) manifold $M = \lbrace (x,0,0,x^3) \in \bbk^{p+1} : x\in\bbk\rbrace$, which has order of tangency  2 with the image of $i(a,b,c) = (a,b,c,0)$.
Yet according to Lemma \ref{lemma_tangencia}, and since we have seen $g$ must be quadratic, this order of tangency should be equal to $\tau(g) = 1$.
This is a contradiction, and therefore, $C_3$ is not an augmentation.
\end{proof}

The analytic stratum of the OPSU is too coarse to properly detect augmentations on its own: to fix Theorem 4.6 in \cite{houstonaug2} we need extra hypothesis on how exactly is the OPSU made trivial along certain parameters.

Assume that $\mononp f$ admits an OPSU, $F\colon(\bbk^{n+1},0)\to(\bbk^{p+1},0)$ with analytic stratum of positive dimension $d = \dim_\bbk\tilde{\tau}(F) \geq 1$.

Define the inclusion $i\colon(\bbk^p,0)\to(\bbk^{p+1},0)$ given by $i(X) = (X,0)$ and, for each $1 \leq s\leq d$, define the projection over the first $p+1-s$ coordinates as $\pi_s\colon(\bbk^{p+1-s}\times\bbk^s,0)\to (\bbk^{p+1-s},0)$

Now, for each $1\leq s\leq d$ we can find $\psi_s\in\text{Diff}_0(p), \phi_s\in\text{Diff}_0(n)$ such that $\psi_s\circ F\circ \phi_s = H_s\times \Id_s$ for some stable map-germ $H_s$.
For each such $\psi_s$ we can define a {\it trivializer} of degree $s$ of $F$ as a composition $T_s(F) = \pi_s\circ\psi_s\circ i$.
These are not uniquely defined for each $s$.
If $\dim_\bbk\tilde\tau \left(T_{s}(F)\right) \geq p-s $, then following Houston's original proof, we can build the following diagram:

\begin{equation}\label{eq_houston_diagram}
\xymatrix{
(\bbk^{n+1 - s},0)\ar[r]^{H_s} & (\bbk^{p+1 - s},0)\ar[r]^\Psi & (\bbk^{p-s}\times\bbk,0)\ar[r]^(.6){pr} & (\bbk,0) \\
(\bbk^{n+1 - s}\times\bbk^s,0) \ar[u] \ar[r]^{H_s\times\Id_s} & (\bbk^{p+1 - s}\times\bbk^s,0)\ar[u]^{\pi_s} & &  \\
(\bbk^{n+1},0) \ar[r]^F\ar[u]^{\phi_s} & (\bbk^{p+1},0)\ar[u]^{\psi_s} & & \\
(\bbk^n,0) \ar[r]^f\ar[u]& (\bbk^p,0)\ar[u]^i \ar[r]^\Phi & (\bbk^{p-s}\times\bbk^s,0)\ar[uuu]^\gamma &
   }
\end{equation}
for some map-germ $\gamma\colon(\bbk^p,0)\to(\bbk^{p+1-s},0)$ of the form:
$$\gamma(X_1,\ldots,X_{p-s},X_{p-s+1},\ldots,X_p) = (X_1,\ldots,X_{p-s},g(X_{p-s+1},\ldots,X_{p}))$$
with $g\in\ofu_s$, $\Psi,\Phi$ germs of diffeomorphisms and $pr$ the natural projection on the last component.

\begin{teo}\label{thm_augcar_fix}
Assume $\mononp f$ is a smooth map-germ that admits an OPSU, $F$.
Then, $f$ is an augmentation if and only if $F$ admits a trivializer $T_s(F)$ of degree $1\leq s \leq \dim_\bbk\tilde{\tau} (F)$ such that $\dim_\bbk\tilde{\tau}(T_s(F))\geq p-s$

\end{teo}
\begin{proof}
Since $\dim_\bbk\tilde{\tau}\left(T_s(F)\right)\geq p-s$, we can build a diagram like (\ref{eq_houston_diagram}).


Now we follow a similar argument to the one in Houston's proof:  $H_s$ is transversal to $\pi_s\circ\psi_s\circ i$ and so $d(pr\circ\Psi\circ H)(0) \neq 0$.
Also, $d(pr\circ\Psi\circ\pi_s\circ\psi\circ i)(0) = 0$.
For $\lambda$ near $0$, $(pr\circ\Psi\circ H)^{-1}(\lambda)\cong\bbk^{n-d}$ and $(pr\circ\psi)^{-1}(\lambda)\cong\bbk^{p-d}$.

Define $h_\lambda = H|(pr\circ\psi\circ H)^{-1}(\lambda)\colon(\bbk^{n-d},0)\to(\bbk^{p-d},0)$, then $\mathcal H(x,\lambda) = (h_\lambda(x),\lambda)$ is an OPSU for $h_0$.
Since the largest square of (\ref{eq_houston_diagram}) is a transverse fibre square, $f$ is $\A$-equivalent to $(x,z)\mapsto(h_{g(z)}(x),z)$, i.e. $f\sim_\A A_{\mathcal H,g}(h_0)$.

Assume now that $f$ is an augmentation.
Then it is $\A$-equivalent to some $A_{H,g}(g)$ for some smooth map-germ $h\colon(\bbk^{n-s},0)\to(\bbk^{p-s},0)$ that admits an OPSU $H(x,\lambda) = (h_\lambda(x),\lambda)$, and some augmenting function $g\colon(\bbk^s,0)\to(\bbk,0)$.

$A_{H,g}(h)$ admits the natural OPSU:
$$A(x,z,\lambda) = (h_{g(z)+\lambda}(x),z,\lambda)$$
Define $\tilde\phi\colon(\bbk^{n-s}\times\bbk \times \bbk^s,0)\to(\bbk^{n-s}\times\bbk^s \times \bbk,0)$ and $\tilde\psi\colon(\bbk^{p-s}\times\bbk^s\times \bbk,0)\to(\bbk^{p-s}\times\bbk \times \bbk^s,0)$ by $\tilde\phi(x,\lambda,z) = (x,z,\lambda - \phi(z))$ and $\tilde\psi(X,Z,\Lambda) = (X,\Lambda +g(Z),Z)$.
Then $H\times\Id_s = \tilde\psi\circ A\circ \tilde\phi$, and $\pi_s\circ\tilde\psi\circ i(X, Z) = (X, g(z))$, hence $\tilde\tau(\pi_s\circ\tilde\psi\circ i) = p-s$.


\end{proof}

\begin{rem}
In Houston's original proof, the second implication was obtained by considering that every augmentation is $\A$-equivalent to the augmentation of a map-germ via a minimal OPSU, in the sense that the OPSU is not $\A$-equivalent to the trivial extension of any other stable map-germ.

As we saw, the OPSU of $C_3$ has analytic stratum of dimension 1 since it is $\A$-equivalent to the trivial extension of the cross-cap.
Any augmentation of $C_3$ is therefore a counterexample to the last statement.


\end{rem}

\begin{ex}
The following examples illustrate two use cases for the previous theorem:
\begin{enumerate}
\item[i)] $C_3$ is not an augmentation because the only possible trivializer is a cuspidal edge (seen at the beginning of this section) and $\dim_\bbk\tilde{\tau}(T_1(F))=1<3-1=2.$
\item[ii)] Now consider the augmentation of $C_3$ given by $f(x,y,z)=(x,y^2,x^3y+xy^3+z^3y,z)$. The natural OPSU $F(x,y,z,\lambda)=(x,y^2,x^3y+xy^3+z^3y+\lambda y,z,\lambda)$ is $\A$-equivalent to $(y^2,\lambda y, \lambda,x,z)$ and so $\dim_\bbk\tilde{\tau} (F)=2$. The trivializer of degree 2 is given by $T_2(F)(X,Y_1,Y_2,Z)=(Y_1,Y_2,X^3+XY_1+Z^3)$ and $\dim_\bbk\tilde{\tau}(T_2(F))=1<4-2=2$. However, $T_1(F)(X,Y_1,Y_2,Z)=(Y_1,Y_2,X^3+XY_1+Z^3,X)$ and $\dim_\bbk\tilde{\tau}(T_1(F))=3\geq 4-1=3$, and so $f$ satisfies the condition in Theorem \ref{thm_augcar_fix} to be an augmentation.
\end{enumerate}
\end{ex}

This characterization can be interpreted geometrically in the following way:


\begin{coro}
A map-germ $f$ is an augmentation if and only if it admits $F$ with $\dim_\bbk\tilde\tau(F) = d \geq 1$ and a trivializer of some degree $1\leq s\leq d$ such that $\tilde\tau(T_s(F))$ is transversal to any axis from $\bbk^{p-d+1}$.
\end{coro}
\begin{proof}
If $\tilde\tau(T_s(F))$ is transversal to any axis, this implies $\dim_\bbk\tilde\tau(T_s(F)) \geq p-d$ and therefore $f$ is an augmentation by Theorem \ref{thm_augcar_fix}.

If $f\colon \bbk^{n'} \to \bbk^{p'}$ is an augmentation then it is equivalent to a germ $ \bbk^{n-d}\times\bbk^d\to\bbk^{p-d}\times\bbk^d$ of the form $(h_{g(z)}(x),z)$ with $(h_\lambda(x),\lambda)$ a stable unfolding of $h_0$.
Therefore it admits the OPSU $F(x,z,\lambda) =  (h_{g(z)+\lambda}(x),z,\lambda)$.
Setting $\phi(x,\lambda,z) = (x,z,\lambda-g(z))$ and $\psi(X,Z,\Lambda) = (X,\Lambda + g(Z),Z)$, we have:
$$\psi\circ F\circ \phi(x,\lambda,z) = (h_\lambda(x),\lambda,z)$$
and $T_s(F) (X,Z) = \pi\circ\psi\circ i (X,Z) = (X,g(Z))$ is transversal to the $p+1$ axis.
\end{proof}


\section{Unfoldings and Augmentations}

Our aim in this section will be understanding how the choice of OPSU affects de $\A$-class of the resulting augmentation.
By Proposition \ref{prop_opsu}, when the $\A_e$-codimension of the augmented map-germ is 1 and the augmenting function is quasi-homogeneous with isolated singularity, the choice is not relevant.
But the proof of this result relies heavily on the fact that, in codimension 1, every OPSU is also a versal unfolding.
Hence, in this case all OPSUs are equivalent as unfoldings after a pull-back with a diffeomorphism.
This setting is too optimistic for higher $\A_e$-codimension map-germs, since it is a very strong condition.

In \cite{manciniruas}, Mancini and Ruas define the following equivalence relation:

\begin{definition}
Let $\phi\colon(\bbk^n,0)\to(\bbk^d,0)$ be a smooth map-germ.
Two map-germs $f,g\colon(\bbk^n,0)\to(\bbk^p,0)$ are said to be {\it $\phi$-equivalent} if there exists a commutative diagram of the form:
\[
\xymatrix{
(\bbk^n,0) \ar[r]^(.4){(f,\phi)}\ar[d]^(.45)h & (\bbk^p\times\bbk^d,0)\ar[d]^(.45)k \ar[r]^(.6)\pi &(\bbk^d,0)\ar[d]^(.45)l \\
(\bbk^n,0) \ar[r]^(.4){(g,\phi)} & (\bbk^p\times\bbk^d,0) \ar[r]^(.6)\pi &(\bbk^d,0)
}
\]
with $h,k,l$ germs of diffeomorphism and $\pi$ the natural projection.
\end{definition}

This equivalence relation is given by a subgroup of $\A$, denoted by $G^\phi$ and given by pairs of diffeomorphisms $(h,k)\in\text{Diff}_0(n)\times\text{Diff}_0(p+d)$ such that for some $l\in\text{Diff}_0(d)$, $l\circ \phi\circ h^{-1} = \phi$, $\pi\circ k = l\circ \pi$. It acts over map-germs of the form $\mnp f \bbk$ as $(h,k)\cdot f = \pi\circ k\circ (f,\phi)\circ h^{-1}$.
If $f$ and $\phi$ are as in the definition above, $TG_e^\phi f$ is defined as the set of $\sigma\in\theta(f)$ such that there exist $\xi\in\theta_n$, $\eta\in\theta_{p+d/d}$ and $\mu\in\theta_d$ so that:
\begin{equation*}
\left\lbrace\begin{split}
\sigma & =  df(\xi) + \eta(f,\phi)\\
0 & =  d\phi(\xi) + \mu\circ\phi
\end{split}\right.
\end{equation*}
Here  $\theta_{p+d/d}$ denotes the space of vector fields in $\theta_{p+d}$ with last $d$ components equal to 0.
The $G_e^\phi$-codimension of $f$ is defined as the codimension of $TG_e^\phi f$ as a subspace of $\theta(f)$, and denoted by $\phicod (f)$.

\medskip

Consider now two OPSUs, $F(x,\lambda) = (\bar{f}(x,\lambda),\lambda), F'(x,\lambda) = (\bar f'(\lambda,x),\lambda)$, not necessarily of the same map-germ.
Given the function $\phi\colon(\bbk^n\times\bbk,0)\to(\bbk,0)$ of the form $\phi(x,\lambda) = \lambda$, we can now ask if $\bar f$ and $\bar f'$ are $\phi$-equivalent.
In this case, the diffeomorphisms of $G^\phi_e$ are pairs $(h,k)$ with $h(x,\lambda) = (\bar{h}_\lambda(x),l(\lambda))$ and $k(X,\Lambda) = (k_\Lambda(X),l(\Lambda))$, being $l\colon(\bbk,0)\to(\bbk,0)$ a diffeomorphism, and $h_0(x), k_0(X)$ diffeomorphisms in $\bbk^n,\bbk^p$ respectively.

Notice that this case of $\phi$-equivalence generalizes equivalence as unfoldings after a pull-back: instead of using unfoldings of the identity, here one is allowed to use unfoldings of any given diffeomorphism.
In particular, this equivalence relation allows us to compare the different OPSUs of two $\A$-equivalent map-germs.

The following result can be found as Theorem 7.24 in \cite{dimcatopics} for the case that $\bbk = \bbc$, and it states that in the holomorphic case, $\R$-equivalence and $\mathscr K$ equivalence are equal in the orbits of quasi-homogeneous functions.
It is extended in \cite{takahashi_requiv} to the real case, giving up to two different equivalence classes depending on the sign of the function.

\begin{prop}[\cite{takahashi_requiv}]\label{prop_tak_requiv}
If $\bbk = \bbc$ and $g,g'\in\ofu_d$ are $\mathscr K$-equivalent and one of them is quasi-homogeneous, then they are $\R$-equivalent.
In the real case, either $g\sim_\R g'$ or $g\sim_\R -g'$.
\end{prop}

\begin{rem}
From now on we will be working over $\bbk = \bbc$ for simplicity, but all results about equivalence classes have a real counterpart up to a finite number.
\end{rem}

\begin{teo}\label{thm_phi_aug}
Let $F(x,\lambda) = (f_\lambda(x),\lambda)$ and $F'(x,\lambda) = (f'_\lambda(x),\lambda)$ be OPSUs of $f,f'\colon(\bbc^n,0)\to(\bbc^p,0)$ respectively.
If $f_\lambda,f'_\lambda$ are $\phi$-equivalent with $\phi(x,\lambda) = \lambda$, then for any $g\in\ofu_d$ that is $\R$-equivalent to a quasi-homogeneous function:
$$A_{F,g}(f) \sim_\A A_{F',g}(f')$$
\end{teo}

\begin{proof}
By hypothesis, there exist germs of diffeomorphisms $h\in\theta_{n+1}$, $k\in\theta_{p+1}$, $l\in\theta_1$ such that the following diagram commutes:
\[
\xymatrix{
(\bbc^{n}\times\bbc,0) \ar[r]^(.5){F}\ar[d]^(.45)h & (\bbc^p\times\bbc,0)\ar[d]^(.45)k \ar[r]^(.6)\pi &(\bbc,0)\ar[d]^(.45)l \\
(\bbc^{n}\times\bbc,0) \ar[r]^(.5){F'} & (\bbc^p\times\bbc,0) \ar[r]^(.6)\pi &(\bbc,0)
}
\]
Using $(X,\Lambda)$ as variables for $\bbc^p\times\bbc$, we have:
$$\pi\circ k(X,\Lambda) = l\circ\pi(X,\Lambda) = l(\Lambda)$$
and so we can write $k = (\bar{k},l)$ for some $\bar{k}\colon(\bbc^p\times\bbc,0)\to(\bbc^p,0)$.
Similarly, if $\tilde{\pi}\colon(\bbc^n\times\bbc,0)\to(\bbc,0)$ is the natural projection, then:
\begin{equation*}
\begin{split}
\tilde{\pi}\circ h(x,\lambda) &= \pi \circ F'\circ h (x,\lambda) \\
& =  l\circ \pi\circ F (x,\lambda)\\
& = l(\lambda)
\end{split}
\end{equation*}
and so we can write $h = (\bar h, l)$ for some $\bar h\colon(\bbc^n\times\bbc,0)\to(\bbc^n,0)$.
Since $F'\circ h = k\circ F$, we have:
\begin{equation}\label{eq_def_opsus}
f'_{l(\lambda)}\circ \bar{h} (x,\lambda) = \bar{k}(f_\lambda(x),\lambda)
\end{equation}

Define $\tilde{F'}(x,\lambda) = (f'_{l(\lambda)}(x),\lambda)$, which is an OPSU of $f$.
First we see that $A_{F',g}(f)$ is $\A$-equivalent to $A_{\tilde{F'},g}(f)$.
Notice that $A_{\tilde{F'},g}(f) = A_{F',l\circ g}(f)$.
Since $g$ is $\R$-equivalent to a quasi-homogeneous function, we can apply Proposition \ref{prop_tak_requiv} to see that  $l\circ g$ is $\R$-equivalent to $g$.
So by Proposition \ref{prop_requiv_augm_aequiv} the augmentation $A_{\tilde F',g}(f) = A_{F',l\circ g}(f)$ is $\A$-equivalent to $A_{F',g}(f)$.

It only rests to see that $A_{\tilde{F'},g}(f)$ is $\A$-equivalent to $A_{F,g}(f)$.
Define $Ah(x,z) = (\bar{h}(x,g(z)),z) \in\theta_{n+d}$ and $Ak(X,z) = (\bar{k}(X,g(z)),z)\in\theta_{p+d}$, which are germs of diffeomorphism.
Then, using Equation \ref{eq_def_opsus}:
$$A_{\tilde{F'},g}(f)\circ Ah = Ak \circ A_{F,g}(f)$$

\end{proof}

Hence, the number of $\phi$-equivalence classes (with $\phi(x,\lambda) = \lambda$) of the different OPSUs of a given map-germ, determines the number of different augmentations that can be generated via a given augmenting function.
In particular, the choice of OPSU will not affect the $\A$-class of an augmentation if we can prove that all the OPSUs of a map-germ lie in the same $\phi$-orbit.

This is precisely the case when the base germ has $\A_e$-codimension 1, which was proven by Cooper, Mond and Wik Atique in \cite{coopermondwik} for Morse augmenting functions, and then extended in \cite{augcod1morse} to any quasi-homogeneous function with isolated singularity (see Proposition \ref{prop_opsu}).
Now we can drop the isolated singularity hypothesis, obtaining:

\begin{coro}\label{coro_opsu}
If $\aecod (f) = 1$ and $g$ is $\R$-equivalent to a quasi-homogeneous function, then the $\A$-class of the augmentation does not depend on the choice of OPSU.
\end{coro}

\begin{rem}
From this point on, we will only be concerned about $\phi$-equivalence in the case that $\phi(x,\lambda) = \lambda$, i. e., the projection over the last component.
We will omit this information, and just write $\phi$-equivalence.

We will say that two OPSUs $F(x,\lambda) = (f_\lambda(x),\lambda), F'(x,\lambda) = (f_\lambda'(x),\lambda)$ are $\phi$-equivalent if $f_\lambda, f_\lambda'$ are $\phi$-equivalent.
When we talk about the $\phi$-codimension of $F$, we are referring to $\phicod (f_\lambda)$.

\end{rem}

The aim in the following sections will be finding conditions to ensure that all OPSUs of a given map-germ lie on the same $\phi$-equivalence class.
But having a finite amount of classes is also useful: in particular, it helps to determine when an augmentation is simple.

\begin{teo}\label{thm_suf_simple}
Let $\mononp f$ be a smooth germ that admits an OPSU $F\colon(\bbk^n\times\bbk,0)\to(\bbk^p\times\bbk,0)$ and $g\colon(\bbk^d,0)\to(\bbk,0)$ a smooth function.
If:
\begin{enumerate}[noitemsep]
  \item $f$ is simple
  \item $g$ is simple
  \item There is a finite number of $\phi$-equivalence classes in which all OPSUs of $f$ lie.
  \item Every deformation of $A_{F,g}(f)$ is the augmentation of some deformation of $f$ via some deformation of $g$.
\end{enumerate}
Then, $A_{F,g}(f)$ is simple.
\end{teo}
\begin{proof}
Since all singularities that appear when deforming $A_{F,g}(f)$ are augmentations of some deformation of $f$, we only need to compute all the possible ways to augment all deformations of $f$ to obtain deformations of $A_{F,g}(f)$.

Since $f$ is simple, there must be a finite number of $\mathcal{A}$-classes in which its deformations lie.
Each of these deformations admits a finite number of OPSUs up to $\phi$-equivalence.
There is only a finite number of $\R$-equivalence classes in the deformations of $g$, hence there is only a finite number of possible augmentations in which the deformations of $A_{F,g}(f)$ lie in.

\end{proof}

\begin{rem}\label{F4simple}
In the above theorem, condition (4) can be changed for
\begin{enumerate}
\item[(4')] Every deformation of $A_{F,g}(f)$ which is not the augmentation of some deformation of $f$ via some deformation of $g$ is simple.
\end{enumerate}
And the same proof yields that $A_{F,g}(f)$ is simple.
\end{rem}

\begin{ex}
The family $S_k\equiv (x,y^2,y^3+ x^ky)$ from \cite{mond2en3} is simple, since its versal unfolding is of the form $(x, y^2, y^3 + (x^k + \sum_{i=0}^{k-2} \lambda_ix^i)y, \lambda)$ and all the singularities that appear are augmentations of deformations of $(y^2,y^3)$ via deformations of $x^k$.
\end{ex}

Notice that as a direct corollary, we can also find some information about where the lack of modality might lie:

\begin{coro}\label{lackofsimplicity}
If $f$ and $g$ are simple, the number of $\phi$-equivalence classes of $f$ is finite, and $A_{F,g}(f)$ is not simple, then the lack of simplicity comes from singularities that are not augmentations or are augmentations of different map-germs.
\end{coro}

\begin{rem}
Suppose $f,g$ are simple and $\nae f = \linsp{\bbk}{\pdat{f_\lambda}{\lambda}{0},\gamma_2,\ldots,\gamma_k}$, and $\frac{\ofu_d}{Jg + g} = \linsp{\bbk}{\tau_1,\ldots,\tau_r}$. We know from \cite{augcod1morse} that a versal unfolding of an augmentation is given by
$$\left(f_{g(z) + \sum_{i=1}^r\tau_i(z)\mu_{i,1}}(x)+\sum_{i=1}^r\sum_{j=2}^k \mu_{i,j}\tau_i(z)\gamma_j(x),\mu\right).$$
Write $\gamma_1 = \pdat{f_\lambda}{\lambda}{0}$.
We can assume $\tau_1 = 1$.
Then, the different parameters in the versal unfolding can be arranged as:
\[
\begin{matrix}
 \gamma_1 &  \gamma_2 & \cdots &  \gamma_k\\
\tau_2 \gamma_1 & \tau_2 \gamma_2 & \cdots & \tau_2 \gamma_k\\
\vdots & \vdots & \ddots & \vdots \\
\tau_r \gamma_1 & \tau_r \gamma_2 & \cdots & \tau_r \gamma_k
\end{matrix}
\]
Making all $\mu_{i,j} = 0$ for $i,j >1$, we see that in this strata we only obtain augmentations of deformations of $f$ and $g$.
If $f$ admits a finite number of $\phi$-equivalence classes for its OPSUs, then the number of singularities that appears here is finite, meaning that if the augmentation is non-simple the modality should be looked for in the strata $\mu_{i,j} \neq 0$ for $i,j > 1$.
\end{rem}

\begin{ex}

The map-germ $F(x,y,z) = \left(z^3 + (x^4+y^4)z,x,y\right)$ can be seen as the augmentation of the simple germ $f(x,z) = \left(z^3+x^4z,x\right)$ via the simple function $y^4$.
$\nae F$ is generated by:
\[
\begin{matrix}
z & xz & x^2z\\
zy & xyz & x^2yz\\
zy^2 & xy^2z & x^2y^2z\\
\end{matrix}
\]
From the versal unfolding of $F$, the deformations in the direction of the bottom right $2$ by $2$ minor are:
$$F_\mu(x,y,z) = \left(z^3 + (x^4+\sum_{\substack{i=1\\j=1}}^2\mu_{i,j} x^iy^j +y^4)z,x,y\right)$$
These deformations can only be obtained as augmentations of $z^3$ and not as augmentations of $f$.
Notice in particular that the stratum $\mu_{1,1}=\mu_{1,2} =\mu_{2,1} = 0$ is where the lack of simplicity is located.
Of course, seeing $F$ directly as the augmentation of $z^3$ via the non-simple function $x^4+y^4$ allows us to directly check that $F$ is non-simple (see \cite{augcod1morse}).

\end{ex}

\section{Weak equivalence and cross-substantial unfoldings }
In the next section we will see that a key property to determine when all the OPSUs of a given map-germ lie on the same $\phi$-equivalence class is substantiality.
Particularly, we will see that this will hold if all its OPSUs are substantial.
To ensure this condition, along this section we will be working with a new, much coarser equivalence relation which preserves substantiality.
In most cases, all OPSUs will lie on the same orbit in this {\it weak} equivalence.

First, we need to construct a ``normal" form for the OPSUs of a given map-germ.
Let $\mononp f$ be a smooth map-germ that admits an OPSU.
In particular, there must exist a $\gamma\in\theta(f)$ such that $(f(x) + \lambda\gamma(x),\lambda)$ is stable.
Denote $k = \aecod(f)$.
Then:

\begin{prop}\label{propo_normal_opsu}
Every OPSU of $f$ is $\phi$-equivalent to one of the form:
\begin{equation} \label{eq_normal_opsu}
\left( f(x) + \left(1 + \sum^{k-1}_{i=1}q_i(\lambda)s_i(f(x))\right)\lambda\gamma(x),\lambda \right)
\end{equation}
for some $q_i\in\ofu_1,s_i\in\mfr_p$.
\end{prop}
\begin{proof}
In \cite{notaopsus}, it is seen that when a map-germ admits an OPSU, one can find a $\bbk$-basis for $\nae f$ of the form:
$$\nae f =\operatorname{Sp}_\bbk\left\lbrace\gamma(x),s_1(f(x))\gamma_(x),\ldots,s_{k-1}(f(x))\gamma(x)\right\rbrace$$
for some polynomials $s_1,\ldots,s_{k-1}\in\mfr_p$.
In particular, this implies that every 1-parameter unfolding is equivalent as an unfolding (and, in particular, $\phi$-equivalent) to one of the form:
\begin{equation}\label{eq_unfeq}
\left( f(x) + \left(\rho_0(\lambda) + \sum^{k-1}_{i=1}\rho_i(\lambda)s_i(f(x))\right)\gamma(x),\lambda \right)
\end{equation}
for some $\rho_0,\ldots,\rho_{k-1}\in\mfr_1$.

Notice that for one of these unfoldings to be stable, it is mandatory that $\rho_0^\prime(0) \neq 0$.
In this case, $\rho_0$ must admit an inverse function germ $\rho_0^{-1}\colon(\bbk,0)\to(\bbk,0)$.
For each $i=0,\ldots,k-1$, define $\tilde q_i\colon(\bbk,0)\to(\bbk,0)$ as $\tilde q_i(\lambda) = \rho_i\circ \rho_0^{-1}(\lambda)$.
Consider the diffeomorphisms given by $h(x) = (x,\rho_0^{-1}(\lambda))$ and $k(X,\Lambda) = (X,\rho_0(\Lambda))$ in $\bbk^{n+1}$ and $\bbk^{p+1}$ respectively.
Notice that $(h,k)\in G_e^\phi$.
Composing with $h$ and $k$ in the source and target respectively, we see that the expression in Equation \ref{eq_unfeq} is $\phi$-equivalent to:
$$\left( f(x) + \left(\lambda + \sum^{k-1}_{i=1}\tilde q_i(\lambda)s_i(f(x))\right)\gamma(x),\lambda \right)$$
Using the fact that $\tilde q_i(0) = 0$, we apply Hadamard's Lemma to see that there exist some ${q_i}\in\ofu_1$ such that $\tilde q_i(\lambda) = \lambda q_i(\lambda)$ for each $i=1,\ldots,k-1$.
Factoring $\lambda$ out, we obtain a expression in the desired form.
\end{proof}

By Theorem \ref{thm_phi_aug}, any augmentation of $f$ will be $\A$-equivalent to the augmentation via an unfolding of the form in Equation \ref{eq_normal_opsu}.
All the OPSUs we consider from now on will be of this form.

Denote by $(x,\lambda)$ the variables for $\bbk^n\times\bbk$, we write $\mfr_\lambda$ the ideal of functions in $\ofu_{n+1}$ multiplied by $\lambda$.
Similarly, denote the variables in $\bbk^p\times\bbk$ as $(X,\Lambda)$, and $\mfr_\Lambda$ analogously to $\mfr_\lambda$.

\begin{definition}
We say that two map-germs $F,G\colon(\bbk^n\times\bbk,0)\to(\bbk^p\times\bbk,0)$ are {\it weak equivalent} if they are $\A$-equivalent via some germs of diffeomorphism $\Phi = (\Phi_1,\ldots,\Phi_{n+1})\in\text{Diff}_0(n+1)$ and $\Psi = (\Psi_1,\ldots,\Psi_{p+1})\in\text{Diff}_0(p+1)$ such that $\Phi_{n+1}\in\m_\lambda\ofu_{n+1}$ and $\Psi_{p+1}\in\m_\Lambda\ofu_{p+1}$.

These diffeomorphisms, whose last component is just the last variable times a unit in the ring of functions, will be called {\it weak} diffeomorphisms.
In particular, if a diffeomorphism is weak, its inverse will also be weak.

\end{definition}

The normal form obtained in Proposition \ref{propo_normal_opsu} does not give any particular hint on what the $s_i$ might be, but as we will see in the next result, having some information on them would be of great use.
Since it is a representative of a non-zero class over the $\K_e$-normal space to $f(x)$, we can consider $\gamma(x)\in\theta(f)$ as a vector field whose components are all zero except for one.
Moreover, we can assume that this component consists of a single monomial.
Without loss of generality, we will assume that this non-zero component is the last one, that is: $\gamma(x) = (0,\ldots,0,x^\alpha)$, with $x^\alpha$ a monomial in $\mfr_n$.

\begin{prop}\label{propo_weak_all_weak}
Assume that we can choose the $s_i\in\mfr_p$ in equation \ref{eq_normal_opsu} so that they do not depend on the last component (i.e., they only depend on $f_1,\ldots,f_{p-1}$).
Then, all OPSUs of $f$ are weak equivalent.
\end{prop}

\begin{proof}
Notice that $\phi$-equivalence entails weak equivalence.
Using our hypothesis and Proposition \ref{propo_normal_opsu}, we can choose some $s_i(X) = s_i(X_1,\ldots,X_p)\in\mfr_p$ not depending on $X_p$ so that every OPSU of $f$ will be weak equivalent to one of the form:
$$F(x,\lambda) = \left( f_1(x),\ldots,f_{p-1}(x), f_p(x) + \left(1 + \sum^{k-1}_{i=1}q_i(\lambda)s_i(f(x))\right)\lambda x^\alpha,\lambda \right)$$
Now, consider the function $H\colon(\bbk^{k-1}\times\bbk \times \bbk,0)\to(\bbk,0)$:
$$H(Y,\Lambda,\lambda) =\Lambda - \left(1 + \sum^{k-1}_{i=1}q_i(\lambda)Y_i\right)\lambda$$
Then, $\dpar{H}{\lambda}(0) \neq 0$ and so by the Implicit Function Theorem there exists a smooth function $g\colon (\bbk^{k-1}\times\bbk,0) \to (\bbk,0)$ such that $H(Y,\Lambda,g(Y,\Lambda))=  0$ and:
$$\dpar{g}{\Lambda}(Y,\Lambda) = \frac{1}{\dpar{H}{\lambda}(Y,\Lambda,g(Y,\Lambda))}\dpar{H}{\Lambda}(Y,\Lambda,g(Y,\Lambda))$$
In particular, $\dpar{g}{\Lambda}(0) \neq 0$. Moreover, since $H(Y,0,g(Y,0))=  0$:
$$0 = \left(1 + \sum^{k-1}_{i=1}q_i(g(Y,0))Y_i\right)g(Y,0)$$
For this to be true in an open neighbourhood of 0, $g(Y,0)$ must be 0 in a small enough neighbourhood. Hence, the germ of $g$ can be expressed as $g(Y,\Lambda) = \Lambda\tilde{g}(Y,\Lambda)$ for some $\tilde{g}\in\ofu_{k}$ with $\tilde{g}(0) \neq 0$.

Define $S(X) = (s_1(X),\ldots,s_{k-1}(X))$ and $\Phi\colon(\bbk^n\times\bbk,0)\to(\bbk^n\times\bbk,0)$ as $\Phi(x,\Lambda) = (x,g(S(f(x)),\Lambda))$, which is a weak diffeomorphism. Then:
$$F\circ\Phi(x,\Lambda) = \left(f_1(x),\ldots,f_{p-1}(x), f_p(x) + \Lambda x^\alpha,\Lambda\tilde g(S(f(x)),\Lambda) \right)$$
Write $\tilde X = (X_1,\ldots,X_{p-1})$. Now $S$ only depends on $\tilde{X}$, hence the map-germ $\psi\colon(\bbk^{p-1}\times\bbk,0)\to (\bbk^{p-1}\times\bbk,0)$ given by:
$$\psi(\tilde X,\Lambda) = \left(\tilde{X},\Lambda\tilde{g}(S(\tilde X),\Lambda)\right)$$
admits an inverse $\psi^{-1}\colon(\bbk^{p-1}\times\bbk,0)\to (\bbk^{p-1}\times\bbk,0)$ of the form:
$$\psi^{-1}(\tilde X,\Lambda) = \left(\psi_1^{-1}(\tilde{X},\Lambda), \Lambda\psi_2^{-1}(\tilde X,\Lambda) \right)$$
Write $\Psi(X,\Lambda) = \left(\psi_1^{-1}(\tilde{X},\Lambda),X_p,\Lambda\psi_2^{-1}(\tilde{X},\Lambda)\right)$, which is a weak diffeomorphism, and we have:
$$\Psi\circ F\circ \Phi = (f_1(x),\ldots,f_{p-1}(x),f_p(x) +\Lambda x^\alpha,\Lambda)$$
\end{proof}

\begin{ex}
Consider the curve $f\colon(\bbk,0)\to(\bbk^2,0)$ given by $f(y) = (y^2,y^5)$.
It is easy to check that $N(f)$ is $\bbk$-generated by the vector field $(0,y)$, and that $\nae f = \operatorname{Sp}_\bbk\left\lbrace (0,y),(0,y^3)\right\rbrace$.
Observe that $s(X_1,X_2) = X_1$ gives us $(0,y^3) = s(f(y))(0,y)$.
Hence, by Proposition \ref{propo_weak_all_weak} all OPSUs of $f$ are weak equivalent.
\end{ex}

\begin{ex}
The map-germ $11_5$ in \cite{rieger2to2} admits the normal form:
$$g(x,y) = (x,y^4+x^2y+xy^2)$$
In this case, $N(g) = \operatorname{Sp}_\bbk\left\lbrace (0,y) \right\rbrace$ and $\nae g = \operatorname{Sp}_\bbk\left\lbrace (0,y),(0,xy) \right\rbrace$
Once again, Proposition \ref{propo_weak_all_weak} ensures all OPSUs of $g$ are weak equivalent.
\end{ex}



Weak equivalence behaves well with substantiality (the case $\delta(F) = 1$ of Definition \ref{def_degsus}) in the following sense:

\begin{prop}\label{propo_weak_subs}
If two OPSUs are weak-equivalent, and one of them is substantial, the other will also be substantial.
\end{prop}
\begin{proof}
From Remark \ref{rem_subs}, a substantial OPSU $F$ admits families of diffeomorphisms $(\Phi_t,\Psi_t)\in\A_F$ with $\Phi_t(x,\lambda) = (\Phi_t^{1}(x,\lambda),\ldots,\Phi_t^{n}(x,\lambda),\lambda e^{t})$
, and $\Psi(X,\Lambda) = (\Psi_1^t(X,\Lambda),\ldots,\Psi_p^t(X,\Lambda),\Lambda e^t)$
such that for each $t\in\bbk$ small enough $(\Phi_t,\Psi_t)\in\A_F$.

If $\tilde F$ is another OPSU and $\tilde\Phi \in\theta_{n+1},\tilde\Psi\in\theta_{p+1}$ are weak diffeomorphisms such that $\tilde F  \circ\tilde \Phi =  \tilde \Psi \circ   F$, then for each $t\in\bbk$ small enough:
 $$(\tilde\Phi\circ\Phi_t\circ\tilde\Phi^{-1}, \tilde\Psi\circ\Psi_t\circ\tilde\Psi^{-1}) \in \A_{\tilde F}$$
This implies, by the Thom-Levine lemma, that the vector field $\eta = \left. \dpar{}{t}\left(\tilde\Psi\circ\Psi_t\circ\tilde\Psi^{-1}\right)\right|_{t=0}$ is liftable.
Applying the chain rule, this is equal to:
$$\eta = d\tilde\Psi\circ \pdat{\Psi_t}{t}{0}\circ\tilde\Psi^{-1}$$

Since $\tilde\Psi$ is weak, its last component can be expressed as $\tilde\Psi_{p+1}(X,\Lambda) = \Lambda\tilde\Psi'_{p+1}(X,\Lambda)$, with $\tilde\Psi'_{p+1}$ a unit in $\ofu_{p+1}$.
Similarly, the last component of $\tilde\Psi^{-1}$ can be expressed as $\tilde\Psi^{-1}_{p+1}(X,\Lambda) = \Lambda\tilde\Psi'^{-1}_{p+1}(X,\lambda)$ with $\tilde\Psi'^{-1}_{p+1}$ a unit in $\ofu_{p+1}$.
Then, the last component of $\eta$ is:
$$\sum_{i=1}^p \Lambda\dpar{\tilde\Psi'_{p+1}}{X_i}(X,\Lambda)\eta_i(\tilde\Psi^{-1}(X,\Lambda)) + \left(\tilde\Psi'_{p+1}(X,\Lambda) + \Lambda\dpar{\tilde\Psi'_{p+1}(X,\Lambda)}{\Lambda}\right)\Lambda\tilde\Psi'^{-1}_{p+1}(X,\Lambda)$$
which is just $\Lambda$ times a unit in $\ofu_{p+1}$.
Since $\Lift(\tilde F)$ is an $\ofu_{p+1}$-module, we can multiply $d\Psi\circ \eta\circ\Psi^{-1}$ by the inverse of this unit and obtain a liftable vector field of $\tilde F$ with only $\Lambda$ in the last component, and so $\tilde F$ is substantial.


\end{proof}

\begin{rem}
The first part of the proof actually establishes a bijection between liftable vector fields of $\A$-equivalent map-germs via the diffeomorphism in the target.
This was already done in Lemma 6.1 from \cite{nishimuralifts}.
Our argument here provides a different proof of the same fact.
\end{rem}

\begin{coro}
If two OPSUs are $\phi$-equivalent and one of them is substantial, then both are substantial.
\end{coro}


\begin{ex}\label{ex_substantial}
Coming back to our examples $f(y) = (y^2,y^5)$ and $g(x,y) = (x,y^4 + x^2y + xy^2)$, all their respective OPSUs are weak equivalent.
On one hand, $f$ admits the OPSU $(y^2,y^5+\lambda y,\lambda)$, which is quasihomogeneous and therefore substantial.
Hence, by Proposition \ref{propo_weak_subs} all OPSUs of $f$ are substantial.

On the other hand, using SINGULAR one can check that the OPSU $(x,y^4 + x^2y + xy^2+\lambda y,\lambda)$ of $g$ is not substantial: hence, by Proposition \ref{propo_weak_subs} none of the OPSUs of $g$ are substantial.
Notice that $g$ is not quasihomogeneous, and so one cannot construct a quasihomogeneous OPSU.
\end{ex}

One condition to guarantee that the hypothesis of Proposition \ref{propo_weak_all_weak} is verified would be the fact that $f_p(x)\gamma(x) \in \tae f$: then, any $s(f(x))\gamma(x)$ lies in the same class over $\nae f$ as $s(f_1(x),\ldots,f_{p-1}(x),0)\gamma(x)$.

Note that this condition is not necessary: it can be checked that the vector field $(y^4+xy^2 + x^2y)\left(\begin{matrix}0 \\ y\end{matrix}\right)$ is not in the $\A_e$-tangent space of $11_5$, and yet we found a basis of the $\A_e$-normal space that satisfied our hypothesis.

\begin{definition}
Consider the OPSU $F(x,\lambda)=(f_1(x),\ldots,f_{p-1}(x),f_p(x)+\lambda x^{\alpha},\lambda)$ (which can always be considered due to the previous discussion). 
We say $F$ is {\it cross-substantial} if it admits a liftable vector field $\eta\in\theta_{p+1}$ of the form $\eta(X,\lambda) = (\eta_1(X,\lambda),\ldots,\eta_p(X,\lambda),X_p)$.
\end{definition}

\begin{lem}\label{propo_cross_condition}
If $F(x,\lambda) = (f(x) + \lambda\gamma(x),\lambda)$ is cross-substantial, then $f_p(x)\gamma(x)\in \tae f$.
\end{lem}

\begin{proof}
Assuming $\gamma(x) = (0,\ldots,0,x^\alpha)$,
let $\eta(X,\lambda) = (\eta_1(X,\Lambda),\ldots,\eta_p(X,\Lambda),X_p)$ be a liftable vector field and $\xi\in\theta_{n+1}$ so that $tF(\xi) + \omega F(\eta)  = 0$.
Looking at the $p+1$ component of this equation, we get: $\xi_{n+1}(x,\lambda) = -\eta_{p+1}\circ F(x,\lambda) = -(f_p(x) + \lambda x^\alpha)$.
Inspecting now the $p$ component allows us to conclude that:
$$\left(d_xf_p + \lambda d_x(x^\alpha)\right)(\xi_1,\ldots,\xi_n)(x,\lambda) + \eta_p(f(x) +\lambda\gamma(x),\lambda) =  x^\alpha(f_p(x) + \lambda x^\alpha)$$

Any other component is of the form $d_xf_i(\xi_1,\ldots,\xi_n)(x,\lambda) + \eta_i\circ F(x,\lambda)=0$.
Take $\tilde{\xi}(x) = \xi(x,0)\in\theta_n$ and $\tilde\eta(X) = \eta(X,0)$ and we have $tf(\tilde\xi)+\omega f(\tilde\eta) = f_p(x)\gamma(x)$.
\end{proof}

Putting all this together we get that cross-substantiality is a sufficient condition for weak equivalence of all OPSUs:
\begin{teo}
If $F(x,\lambda) = (f(x) + \lambda\gamma(x),\lambda)$ is cross-substantial, then all OPSUs of $f$ are weak equivalent.
\end{teo}
\begin{proof}
Since $f$ admits an OPSU, we can pick a basis for $\nae f$ of the form:
$$\nae f =\operatorname{Sp}_\bbk\left\lbrace\gamma(x),s_1(f(x))\gamma(x),\ldots,s_{k-1}(f(x))\gamma(x)\right\rbrace$$
By Lemma \ref{propo_cross_condition}, $f_p(x)\gamma(x)\in \tae f$, and so each $s_i(f(x))\gamma(x)$ lies on the same class over $\nae f$ as $s_i(f_1(x),\ldots,f_{p-1}(f(x)),0)$, providing another basis that satisfies the conditions in \ref{propo_weak_all_weak}
\end{proof}

\begin{prop}
All corank 1 map-germs of $\A_e$-codimension $1$ are cross-substantial.
\end{prop}
\begin{proof}
In \cite{coopermondwik} it is proven that all corank 1 map-germs of $\A_e$-codimension 1 are $\A$-equivalent to a quasi-homogeneous map-germ.
All quasi-homogeneous map-germs admit a substantial OPSU (see discussion after Definition \ref{def_degsus}).
All OPSUs of an $\A_e$-codimension 1 map-germ are versal unfoldings, hence all of them are equivalent as unfolding after a pull-back, and therefore all of them are weak equivalent.
By Proposition \ref{propo_weak_subs}, this implies all of them are substantial.

Apply now Damon's Theorem \ref{thm_damon} to see that for any corank 1 $\mononp f$ of $\A_e$-codimension 1, and any OPSU $F$:
$$1 = \dim_\bbk\frac{\ofu_{p+1}}{d\pi(\Lift(F))+\langle\Lambda\rangle} = \dim_\bbk\frac{\ofu_{p+1}}{d\pi(\Lift(F))} $$
where the second equality comes from the fact that $F$ is substantial.
Now, to have codimension 1 this means that the only non-zero class is that of the constant functions, implying $X_p\in d\pi(\Lift(F))$.
\end{proof}

Finally, cross-substantiality can be ensured by checking a condition over the defining function of the image of $F$:

\begin{prop}
Let $H\colon(\bbk^p\times\bbk,0)\to(\bbk,0)$ be a defining function for the image of the OPSU $F(x,\lambda) = (f(x) + \lambda\gamma(x),\lambda)$.
If $X_p\in J_XH = \operatorname{Sp}_{\ofu_{p+1}}\left\lbrace \dpar{H}{X_1},\ldots,\dpar{H}{X_p}\right\rbrace$, then $F$ is cross-substantial.
\end{prop}

\begin{proof}
If $X_p\in J_XH$, there exist some $s_1,\ldots,s_p\in\ofu_{p+1}$ such that:
$$X_p = \sum_{i=1}^p s_i(X,\Lambda)\dpar{H}{X_i}(X,\Lambda)$$
Write $\dpar{}{X_1},\ldots,\dpar{}{X_p},\dpar{}{\Lambda}$ for the constant vector fields of $\theta_{p+1}$.
Then, the vector field
$$\eta = \dpar{H}{\Lambda}(X,\Lambda)\sum_{i=1}^p s_i(X,\Lambda)\dpar{}{X_i} - X_p\dpar{}{\Lambda}$$
satisfies  $\eta(H) = 0$ and so belongs to $\operatorname{Derlog}(\Delta F) = \operatorname{Lift}(F)$ (see page 283 in \cite{nunomond} for details).
\end{proof}

\begin{ex}\label{ex_cross}
The curve $(y^2,y^5)$ admits the substantial OPSU $F(y,\lambda) = (y^2, y^5 +\lambda y,\lambda)$, which has as a defining equation the polynomial $H(X,Y,\Lambda) = Y^2 + X(X^2+\Lambda)^2$.
Computing the Jacobian of $H$ with respect to $X,Y$, we see that $Y$ belongs to it.
Hence, $F$ is cross-substantial.

On the other hand, it can be checked via SINGULAR that the OPSU of $11_5$ given by $(x,y^4 + xy^2+x^2y+ \lambda y)$ is not cross-substantial, even though all of its OPSUs are weak equivalent.

\end{ex}

\section{Relative spaces and $\phi$-equivalence}

Recall our setting: $f\colon(\bbk^n,0)\to(\bbk^p,0)$ is a smooth map-germ that admits an OPSU $F(x,\lambda) = (f_\lambda(x),\lambda)$.
Denote $\bar{f}(x,\lambda) = f_\lambda(x)$.
We will be concerned with $\phi$-equivalence with $\phi(x,\lambda) = \lambda$.

The space of vector fields in $\theta_{m+1}$ with last component equal to 0 will be denoted by $\theta_{m+1/1}$.
With this notation, the $\phi$-extended tangent space of $\bar{f}$ can be described like this:
\begin{equation*}
TG^\phi_e\bar{f} = \left\lbrace d_x\bar f(\xi(x,\lambda)) + d_\lambda\bar f(\tilde\xi(\lambda)) + \eta(\bar{f}(x,\lambda),\lambda) : \xi\in\theta_{n+1/1},\tilde\xi\in\ofu_1, \eta\in\theta_{p+1/1}\right\rbrace
\end{equation*}
This space contains what is called the \textit{relative $\A_e$-tangent space to $F$}.
To describe it, denote by $t_\text{rel}F$ the restriction of $tF$ to vector fields in $\theta_{n+1/1}$, and by $\omega_{\text{rel}}F$ the restriction of $\omega F$ to $\theta_{p+1/1}$.
The images of these restrictions both lie on the set of vector fields in $\theta(F)$ with last component equal to zero, which can be naturally identified with the space $\theta(\bar{f})$.
Hence we can consider both mappings as being defined in the following way:
\begin{equation*}
\begin{split}
t_\text{rel}F&\colon\theta_{n+1/1}\to\theta(\bar{f}) \\
\omega_\text{rel}F&\colon\theta_{p+1/1}\to\theta(\bar{f})
\end{split}
\end{equation*}
Then, the relative $\A_e$-tangent and relative $\A_e$-normal spaces to $F$ are defined as:
\begin{equation*}
\begin{split}
\tae (F/1) &= t_\text{rel}F(\theta_{n+1/1})+\omega_\text{rel}F(\theta_{p+1/1})\\
\nae (F/1) &= \frac{\theta(\bar{f})}{t_\text{rel}F(\theta_{n+1/1})+\omega_\text{rel}F(\theta_{p+1/1})}
\end{split}
\end{equation*}

We can also define a mapping $t_1\bar f\colon \ofu_1 \to \theta(\bar f)$ given by $t_1\bar f(\tilde\xi) = d_\lambda\bar f(\tilde\xi) = \tilde\xi\cdot d_\lambda\bar f$.
With this notation, $TG_e^\phi\bar f$ can be rewritten as:
$$TG_e^\phi\bar f = t_\text{rel}F(\theta_{n+1/1}) + t_1\bar f(\ofu_1) + \omega_\text{rel}F(\theta_{p+1/1})$$
Hence $\tae (F/1) \subseteq TG_e^\phi\bar{f}$.

The following result allows one to compare the relative version of the tangent space of an unfolding with the normal space of the map-germ it unfolds.
The general version can be found as Lemma 3.5 in \cite{nunomond}: we only write the case for 1-parameter unfoldings.
Note that every vector field in $\theta(f)$ can be naturally considered as a vector field in $\theta(\bar f)$:
\begin{lem}\label{lemma_relative_normal}
Let $\mononp f$ be a smooth map-germ with $F\colon(\bbk^n\times\bbk,0)\to(\bbk^p\times\bbk,0)$ a 1-parameter unfolding.
Let $\gamma_1,\ldots,\gamma_k\in\theta(f)$.
The following are equivalent:
\begin{itemize}
\item  $\nae f = \operatorname{Sp}_\bbk\lbrace\gamma_1,\ldots,\gamma_k\rbrace$.
\item $\nae (F/1) = \operatorname{Sp}_{\ofu_1}\lbrace\gamma_1,\ldots,\gamma_k\rbrace$, through the natural inclusion in $\theta(\bar f)$.
\end{itemize}
\end{lem}

We see that the relative normal space of an unfolding is essentially the normal space of the base map-germ multiplied by functions depending only on the parameter.
This hints at the elements that may be missing in the $\phi$-tangent space of a given OPSU: our next goal will be trying to compute them.
Recall that we denoted by $\mfr_\lambda$ the space of all functions in $\ofu_{n+1}$ multiplied by $\lambda$.
Then, $\mfr_\lambda\theta(\bar{f})$ is the space of all vector fields in $\theta(\bar{f})$ multiplied by $\lambda$, which is an $\ofu_{p+1}$-module via $F$.

\begin{lem}\label{lemma_mlambda}
Assume $\theta(f)$ is finitely generated as an $\ofu_p$-module via $f$.
For any OPSU $F$ in the form of Equation \ref{eq_normal_opsu}, $F$ is substantial if and only if $\mfr_\lambda\theta(\bar{f})\subseteq \tae(F/1)$.
\end{lem}

\begin{proof}
Assume $F$ is substantial.
We will prove the result by applying Nakayama's lemma.
Having an OPSU in the form of Equation \ref{eq_normal_opsu} means that we have a stable $F(x,\lambda) = (\bar{f}(x,\lambda),\lambda)$, with $\bar{f}(x,\lambda) = f(x) + \left(1+\sum_{i=1}^{k-1}q_i(\lambda)s_i(f(x))\right)\lambda\gamma(x)$, so that $\gamma\in\theta(f)$, $s_1,\ldots,s_{k-1}\in\mfr_{p}$ and:
\begin{equation}\label{eq_nae_f}
\nae f = \linsp{\bbk}{\gamma(x),s_1(f(x))\gamma(x),\ldots,s_{k-1}(f(x))\gamma(x)}
\end{equation}
Recall that $\tae(F/1)$ is an $\ofu_{p+1}$-module via $F$.
Since $F$ is an unfolding of $f$ and $\theta(f)$ is finitely generated as an $\ofu_p$-module via $f$, $\mfr_\lambda\theta(\bar{f})$ is finitely generated as an $\ofu_{p+1}$-module via F.
To apply Nakayama's lemma we need to check:
$$\mfr_\lambda\theta(\bar f) \subseteq \tae (F/1) + F^*\mfr_{p+1}\mfr_\lambda\theta(\bar{f})$$
Now, by Lemma \ref{lemma_relative_normal} and Equation \ref{eq_nae_f}, it suffices to check that: $$\operatorname{Sp}_{\mfr_\lambda}\left\lbrace\gamma(x),s_1(f(x))\gamma(x),\ldots,s_{k-1}(f(x))\gamma(x)\right\rbrace \subseteq \tae (F/1) + F^*\mfr_{p+1}\mfr_\lambda\theta(\bar{f})$$
Using the fact that $\bar f$ is equal to $f$ plus an element in $\mfr_\lambda\theta(\bar f)$, we see that:
\begin{align*}
F^*\mfr_{p+1}\mfr_\lambda\theta(\bar f)
& = \linsp{\ofu_{n+1}}{\lambda\bar{f}_1(x,\lambda),\ldots,\lambda\bar{f}_p(x,\lambda),\lambda^2}\theta(\bar f) \\
& = \linsp{\ofu_{n+1}}{\lambda{f}_1(x),\ldots,\lambda{f}_p(x),\lambda^2}\theta(\bar f)
\end{align*}
Hence, we only need to check that $\lambda\gamma(x) \in \tae (F/1) + F^*\mfr_{p+1}\mfr_\lambda\theta(\bar{f})$.
But since $F$ is substantial, there exists $\eta\in\theta_{p+1}$ of the form $\eta(X,\Lambda) = (\tilde{\eta}(X,\Lambda),\Lambda)$ and a $\xi\in\theta_{n+1}$ of the form $\xi(x,\lambda) = (\tilde{\xi}(x,\lambda),\lambda)$ such that $\eta\circ F = dF\circ\xi$, that is:
$$\tilde{\eta}\circ F = d_x\bar f(\tilde\xi) +  d_\lambda\bar{f}(\lambda) =d_x\bar f(\tilde\xi) +  \lambda d_\lambda\bar{f} $$
This means that $\lambda d_\lambda\bar{f} \in\tae(F/1)$.
Derive $\bar{f}$ with respect to $\lambda$ to check that:
$$\lambda d_\lambda\bar f = \lambda\gamma(x) + \lambda\left(\sum_{i=1}^{k-1}\left(q_i(\lambda) + \lambda q_i^\prime(\lambda)\right)s_i(f(x))  \right)\gamma(x)$$
Observe now that since $q_i(0) = 0$ for each $i=1,\ldots,k-1$: $$\lambda\left(\sum_{i=1}^{k-1}\left(q_i(\lambda) + \lambda q_i^\prime(\lambda)\right)s_i(f(x))  \right)\gamma(x)\in\linsp{\ofu_{n+1}}{\lambda^2}\theta(\bar{f}) \subseteq F^*\m_{p+1}\mfr_\lambda\theta(\bar{f}).$$ 
Hence, $\lambda\gamma(x)\in  \tae (F/1) + F^*\mfr_{p+1}\mfr_\lambda\theta(\bar{f})$.

Now, assuming $\mfr_\lambda\theta(\bar{f})\subseteq \tae(F/1)$, this implies $ d_\lambda f_\lambda(\lambda) \in \tae(F/1)$, which means there exist $\tilde{\xi}\in\theta_{n+1/1}$ and $\tilde{\eta}\in\theta_{p+1/1}$ such that:
$$d_x\bar{f}(\tilde{\xi}) + \tilde{\eta}\circ F =  d_\lambda f_\lambda(\lambda)$$
Then, if $\xi(x,\lambda) = (-\tilde{\xi}(x,\lambda),\lambda)$ and $\eta(X,\Lambda) = (\tilde{\eta}(X,\Lambda),\Lambda)$, we have $dF\circ\xi = \eta\circ F$.

\end{proof}

Recall Mather's lemma (see Theorem 6.3 from \cite{nunomond}):

\begin{teo}\label{lemma_mather}
If a Lie group $G$ acts smoothly on a manifold $M$, and $W\subseteq M$ is a connected submanifold, a necessary and sufficient condition for $W$ to be contained in a single $G$-orbit is that:
\begin{enumerate}[noitemsep]
\item for all $x\in W$, $T_xW\subseteq T_x(Gx)$, and
\item the dimension of $T_x(Gx)$ is the same for all $x\in W$.
\end{enumerate}
\end{teo}

We are now in the position to prove:

\begin{teo}\label{thm_all_phi}
If  $\theta(f)$ is finitely generated as an $\ofu_p$-module via $f$ and
all OPSUs of $f$ are substantial, then all of them are $\phi$-equivalent.
\end{teo}

\begin{proof}

First, we check that all OPSUs in the form of Equation \ref{eq_normal_opsu} have the same $\phi_e$-codimension: since $TG_\phi^e\bar f = \tae (F/1) + t_1 \bar f(\ofu_1)$, we can use Lemmas \ref{lemma_relative_normal} and \ref{lemma_mlambda} to see that:
$$\frac{\theta(\bar{f})}{TG_\phi^e \bar f} = \operatorname{Sp}_{\bbk}\left\lbrace s_1(f(x))\gamma(x),\ldots,s_{k-1}(f(x))\gamma(x)\right\rbrace$$
Hence, they are all $d$-$\phi$-determined for some degree $d$ big enough.
For each combination of $q_i\in\ofu_1$, we can define the following family of map-germs:
$$F_a(x,\lambda) = (\bar f_a(x,\lambda),\lambda) = \left(f(x) + \lambda\gamma(x) + a\lambda\sum_{i=1}^{k-1}q_i(\lambda)s_i(f(x))\gamma(x),\lambda\right)$$
The family $\bar f_a$ defines a connected submanifold $M$ of the space of $d$-jets, $J^d(\bbk^{n+1},\bbk^p)$.
For each $a$, we have $T_{\bar f_a}M\subseteq TG_e^{\phi,(d)}\bar{f}_a$, where $TG_e^{\phi,(d)}\bar{f}_a$ is the $\phi$-tangent space to the $\phi$-orbit of $\bar{f}_a$ in $J^d(\bbk^{n+1},\bbk^p)$, consisting of $d$-jets of vector fields in $TG_\phi^e \bar f$.
Since the codimension of the $\phi$-tangent space is constant along $\bar f_a$, so will be the dimension of $TG_e^{\phi,(d)}\bar{f}_a$.
Hence, applying Mather's lemma \ref{lemma_mather}, all the $j^d\bar f_a$ lie on the same $\phi$-orbit on $J^d(\bbk^{n+1},\bbk^p)$.
Since all $\bar{f_a}$ are $d$-determined, this implies they all lie on the same $\phi$-orbit as  $\bar{f}_0(x,\lambda) = f(x) + \lambda\gamma(x)$.
This can be done for any combination of $q_i\in\ofu_1$, and so every OPSU in the form of Equation \ref{eq_normal_opsu} is $\phi$-equivalent to $F_0$.
In particular, by Proposition \ref{propo_normal_opsu} every OPSU is $\phi$-equivalent to $F_0$.
\end{proof}

We can join all the results from this section to give a condition for which all OPSUs of a given map-germ are $\phi$-equivalent:

\begin{teo}\label{teo_opsu_indep}
If $f$ admits an OPSU of the form $ (f(x) + \lambda\gamma(x),\lambda)$ which is both substantial and cross-substantial, then all OPSUs of $f$ are $\phi$-equivalent, and hence the choice of OPSU for augmenting $f$ does not affect the $\A$-class of the augmentation.
\end{teo}
\begin{proof}
By Proposition \ref{propo_weak_all_weak}, as this OPSU is cross-substantial, all OPSUs will be weak-equivalent.
Since one of them is substantial, by Proposition \ref{propo_weak_subs} all of them will be substantial.
Apply now Theorem \ref{thm_all_phi}.
\end{proof}

\begin{ex}\label{ex_cusps_cross}
All the map-germs in the family $f_k(y) = (y^2,y^{2k+1})$ for $k\geq 1$ admit the cross-substantial OPSU $F_k(y,\lambda) = (y^2,y^{2k+1} +\lambda y,\lambda)$, which is also substantial since it is quasi-homogeneous.
Hence, applying Theorem \ref{teo_opsu_indep}, augmentations of $f_k$ do not depend on the chosen OPSU.
\end{ex}

\section{Consequences and examples}\label{sec_examples}

We have seen in Theorem \ref{thm_phi_aug} that if two OPSUs are $\phi$-equivalent with $\phi(x,\lambda) = \lambda$, they will produce the same augmentation.
Hence, the number of $\phi$-equivalence classes of a given map-germ determines the possible  different augmentations that it can produce.

\begin{definition}
A map-germ $f$ is \textit{good for augmentation} if all its OPSUs lie on the same $\phi$-equivalence class.
\end{definition}

\begin{prop} \label{propo_good_augm}
If $f$ is good for augmentation and $g\sim_\A f$, then $g$ is good for augmentation.
\end{prop}

\begin{proof}
Let $h\colon(\bbk^n,0)\to(\bbk^n,0)$ and $k\colon(\bbk^p,0)\to(\bbk^p,0)$ be two germs of diffeomorphism such that $f = k\circ g\circ h$.
Define $H(x,\lambda) = (h(x),\lambda)$ and $K(X,\Lambda) = (k(X),\Lambda)$, which are $\phi$-equivalence diffeomorphisms.
Given $G,G'$ any two OPSUs of $g$, we have that both $K\circ G\circ H$ and $K\circ G'\circ H$ are OPSUs of $f$: hence, they are $\phi$-equivalent.
\end{proof}

\begin{prop}
If $f\colon(\bbk^n,0)\to(\bbk^p,0)$ is good for augmentation, then the $\A$-class of the augmentation only depends on  the $\R$-class of the augmenting function and the $\A$-class of $f$.
\end{prop}
\begin{proof}
The fact that all augmentations of a given $f$ do not depend on the choice of the OPSU is given by Theorem \ref{thm_phi_aug}, and the fact that the choice of augmenting function only affects up to $\R$-equivalence comes from Proposition \ref{prop_requiv_augm_aequiv}.

Given any $g\colon(\bbk^n,0)\to(\bbk^p,0)$ in the $\A$-equivalence class of $f$, let $h,k$ be diffeomorphisms such that  $f = k\circ g\circ h$, and define $H(x,\lambda) = (h(x),\lambda)$ and $K(X,\Lambda) = (k(X),\Lambda)$.
By Proposition \ref{propo_good_augm}, $g$ is also good for augmentation.
Denote by $A_sg$ a representative in the $\A$-class of the augmentation of $g$ via a function $s\in\ofu_d$.
Then, $K\circ A_sg\circ H$ is the augmentation of $f$ via $s$.
\end{proof}

\begin{coro}\label{coro_subs_cross_simple}
If $\mononp f$ is simple, $F(x,\lambda) = (f(x)+\lambda\gamma(x),\lambda)$ is a substantial and cross-substantial OPSU and $g\in\ofu_d$ is simple, and all deformations of $A_{F,g}(f)$ are augmentations of $f$ via deformations of $g$, then $A_{F,g}(f)$ is simple.
\end{coro}
\begin{proof}
By Theorem \ref{teo_opsu_indep}, $f$ is good for augmentation, and so we can apply Theorem \ref{thm_suf_simple} to ensure that the resulting augmentation is simple.
\end{proof}

Lastly, we give a list of examples with different properties in which we can apply the results developed throughout the paper.

From \cite{augcod1morse}, we know that if the augmenting function is simple and the augmented map-germ has $\A_e$-codimension 1, then the augmentation will be simple.
Similarly, if the map-germ is simple and the augmenting function is a Morse function, the augmentation will be simple.
Moreover, if either the map-germ or the augmenting function are not simple, then the augmentation will also fail to be simple.

\begin{ex}
Let $f\colon(\bbk,0)\to(\bbk^2,0)$ be the curve $f(y) = (y^2,y^5)$.
As we saw in Examples \ref{ex_substantial} and \ref{ex_cross} above, the OPSU $F(y,\lambda) = (y^2, y^5 + \lambda y)$ is substantial and cross-substantial: hence by Propositions \ref{propo_weak_all_weak}, \ref{propo_weak_subs} and Theorem \ref{thm_all_phi}, $f$ is good for augmentation and we do not need to consider any other OPSU.

This is a simple map-germ with $\aecod(f) = 2$, and when augmented via the function $g(z) = z^3$ with $\tau(g) = 2$, we obtain the map-germ $F_4$ from \cite{mond2en3}: it is a simple map-germ, yet the codimensions lie outside of the scope of the theorems mentioned from \cite{augcod1morse}.
In the versal unfolding of $F_4$ we can find singularities of type $C_3\equiv(x,y^2,xy^3+x^3y)$, which we have seen are not augmentations, so we cannot use Corollary \ref{coro_subs_cross_simple} either. However, by Remark \ref{F4simple}, since the $C_3$ singularities are simple, $F_4$ is simple.

Augmenting $f$ via the function $z^4$, which is simple and has Tjurina number equal to 3, yields a map-germ which we will call $F_6\equiv\left( x,y^2, y^5 + x^4y\right)$.
In its versal unfolding, besides simple, non-augmentation singularities of type $C_3$ and $C_4$, we can find the family:
$$F_6^a(x,y) = \left(x,y^2,y(y^4 + ax^2y^2 + x^4)\right)$$
We can use Theorem \ref{thm_augcar_fix} to check that these are not augmentations and the classification in \cite{mond2en3} to see that $F_6^a$ is non-simple for each $a\in\bbk$, which implies that $F_6$ is not simple, i.e. it is not equivalent to any germ in Mond's classification.
\end{ex}

\begin{ex}\label{115z3}
We saw in Example \ref{ex_substantial} that the map-germ $11_5\equiv(x,y^4+x^2y + xy^2)$ from \cite{rieger2to2} does not admit any substantial OPSU, since all of them are weak equivalent and it admits a non-substantial OPSU.

This map-germ of $A_e$-codimension $2$ is not an augmentation and it can be augmented via $z^3$ to the non-simple map-germ $A_{F,z^3}(11_5)(x,y,z) = (x,y^4+x^2y + xy^2 +z^3y,z)$ (see \cite{marartari}).
It has $\A_e$-codimension 4, and all the singularities that appear in its bifurcation diagram are augmentations of Beaks, Lips, Swallowtails and $11_5$ (all of which appear in the bifurcation diagram of $11_5$), except for the singularities in the strata:
$$F_a(x,y,z) = (x,y^4+x^2y + xy^2 +z^3y + azxy,z)$$
which, applying Theorem \ref{thm_augcar_fix}, are not augmentations. Then, by Corollary \ref{lackofsimplicity}, $a$ is the only modal parameter in the versal unfolding of $A_{F,z^3}(11_5)$.
\end{ex}


\begin{ex}\label{52t3}

The map-germ $5_2\equiv (x,y,z^5 + xz + y^2z^2 + yz^3)$, from Marar and Tari's list of simple map-germs from $\bbc^3\to\bbc^3$ (\cite{marartari}), is not quasihomogeneous, but admits an OPSU $F(x,y,z,\lambda) = (x,y,z^5 + xz + y^2z^2 + yz^3 + \lambda z^2,\lambda)$.
$F$ is not substantial, but it is cross-substantial: hence, by Propositions \ref{propo_cross_condition}, \ref{propo_weak_subs}, none of its OPSUs are substantial. This means we do not know if there are a finite number of $\phi$-equivalence classes of OPSUs.

We can consider the augmentation of $5_2$ by $t^3$, $(x,y,z^5 + xz + y^2z^2 + yz^3 + t^3z^2,t)$ which has $\mathscr A_e$-codimension 4. A versal unfolding is given by $(x,y,z^5+ xz + y^2z^2 + yz^3 + t^3z^2+\lambda_1z^2+\lambda_2z^3+\lambda_3tz^2+\lambda_4tyz^2,t,\lambda)$. We know by \ref{lackofsimplicity} that if it is not simple, the modality must lie in the stratum $\lambda_1=\lambda_2=\lambda_3=0$. In fact, calculations in SINGULAR using Damon's theorem from \cite{damonakv} show that $(x,y,z^5+ xz + y^2z^2 + yz^3 + t^3z^2+atyz^2,t)$ has $\mathscr A_e$-codimension 4 for any $a$, and so $A_{F,t^3}(5_2)$ is not simple.

\end{ex}

In \cite{augcod1morse} we conjectured that all corank 1 simple augmentations from $\bbc^4$ to $\bbc^4$ are the ones given in Table \ref{table_44}. These are the only possible augmentations of $\mathscr A_e$-codimension 1 germs and by Morse functions. It remained to see what happened when augmenting higher codimension germs by non-Morse functions. From Example \ref{115z3}, we can deduce that $A_{F,z^3+t^2}(11_5)(x,y,z,t)$ is not simple and, by adjacencies, $A_{F,R(z,t)}(11_{2k+1})(x,y,z,t)$ is not simple for any $k\geq 2$ if $R$ is an $A_2$ singularity or worse (non-Morse). On the other hand, from Example \ref{52t3}, $A_{F,t^k}(5_2)(x,y,z,t)$ is not simple for $k\geq 3$, and, since $5_3$ is adjacent to $5_2$, $A_{F,t^k}(5_3)(x,y,z,t)$ is not simple for $k\geq 3$. This proves the following

\begin{teo}
The classification of corank 1 simple augmentations from $\bbc^4$ to $\bbc^4$ is given in Table \ref{table_44}.
\bgroup
\def\arraystretch{1.2}%
\begin{table}[h]
\begin{tabular}{clcl}
\hline
Type    & \multicolumn{1}{c}{Normal Form}    & \multicolumn{1}{c}{$\A_e$-codimension} &            \\ \hline
$3_{P}$ & $(x,y,z,t^3+P(x,y,z)t)$            & $\mu(P)$                               & \\
$4_Q$   & $(x,y,z,t^4+xt+Q(y,z)t^2)$         & $\mu(Q)$                               &            \\
$4_k^2$ & $(x,y,z,t^4+(x^k+y^2+z^2)t+xt^2)$   & $k$                                    & $k\geq 2$  \\
$5_k$   & $(x,y,z,t^5+xt+yt^2+z^kt^3)$       & $k-1$                                  & $k \geq 1$ \\
$5^2$   & $(x,y,z,t^5+xt+(y^2+z^2)t^2+yt^3)$ & $2$                                    &            \\
$5^3$   & $(x,y,z,t^5+xt+z^2t^2+yt^3)$       & $3$                                    &            \\ \hline
\end{tabular}
\medskip
\caption{$P,Q$ are of type $A_k,D_k,E_6,E_7,E_8$.}
\label{table_44}
\end{table}
\egroup
\end{teo}

\printbibliography

\end{document}